\documentclass[10pt]{amsart}

\usepackage{graphicx}
\usepackage{amsmath}
\usepackage{amscd}
\usepackage{amsthm}
\usepackage{amsfonts}
\usepackage{amssymb}
\usepackage{hyperref}
\usepackage[all]{xy}

\newcommand{\PP}{\mathbb{P}}
\newcommand{\ZZ}{\mathbb{Z}}

\newcommand{\AAA}{\mathcal{A}}
\newcommand{\BBB}{\mathcal{B}}
\newcommand{\QQQ}{\mathcal{Q}}
\newcommand{\RRR}{\mathcal{R}}

\newcommand{\pp}{\mathbf{p}}
\newcommand{\qq}{\mathbf{q}}
\newcommand{\rr}{\mathbf{r}}
\newcommand{\sS}{\mathbf{s}}
\newcommand{\tT}{\mathbf{t}}

\newcommand{\eps}{\varepsilon}

\newcommand{\Ncycle}{\#_{\mathrm{cycle}}}
\newcommand{\dcycle}{d_{\mathrm{cycle}}}

\newcommand{\perm}[1]{\mathrm{PAIR}(#1)}
\newcommand{\irr}[1]{\mathrm{PAIR}_{\mathrm{irr}}(#1)}
\newcommand{\std}[1]{\mathrm{PAIR}_{\mathrm{std}}(#1)}
\newcommand{\ClassLab}[1]{\RRR(#1)}
\newcommand{\ClassNon}[1]{\RRR^{\mathrm{non}}(#1)}
\newcommand{\ExtClassLab}[1]{\RRR_{\mathrm{ex}}(#1)}
\newcommand{\ExtClassNon}[1]{\RRR^{\mathrm{non}}_{\mathrm{ex}}(#1)}
\newcommand{\STAR}[1]{\Xi(#1)}

\newcommand{\ARFof}[1]{\mathrm{ONES}(#1)}
\newcommand{\ARFofE}[1]{\mathrm{ONES}_{eq}(#1)}
\newcommand{\ARFofNE}[1]{\mathrm{ONES}_{neq}(#1)}

\newcommand{\Sof}[1]{\sigma_{#1}}
\newcommand{\Mof}[1]{\mathcal{M}_{#1}}

\newcommand{\Yof}[1]{\tilde{\sigma}_{#1}}
\newcommand{\Pof}[1]{\mathcal{P}_{#1}}

\newcommand{\QF}[1]{\mathcal{Q}_{#1}}

\newcommand{\term}[1]{\emph{#1}}
\newcommand{\ol}[1]{\overline{#1}}
\newcommand{\RED}[1]{|_{#1}}
\newcommand{\mtrx}[1]{\left(\begin{matrix} #1 \end{matrix}\right)}
\newcommand{\bmtrx}[1]{\left[\begin{matrix} #1 \end{matrix}\right]}
\newcommand{\LL}[2]{\begin{matrix} #1 \\ #2 \end{matrix}}

\newcommand{\RHScase}[1]{\left\{ \begin{array}{ll} #1 \end{array}\right.}

\newtheorem{thm}{Theorem}[section]
\newtheorem{defn}[thm]{Definition}
\newtheorem{prop}[thm]{Proposition}
\newtheorem{lem}[thm]{Lemma}
\newtheorem{remark}[thm]{Remark}
\newtheorem{cor}[thm]{Corollary}
\newtheorem{move}[thm]{Move}

\newtheorem*{thmPWOR}{Theorem \ref{ThmPWOrderReversing}}

\title[Combinatorial Proof of KZB]{A Combinatorial Proof of the Kontsevich-Zorich-Boissy Classification of Rauzy Classes}
\author[J. Fickenscher]{Jon Fickenscher}

\begin{document}

\begin{abstract}Rauzy Classes and Extended Rauzy Classes are equivalence classes of permutations that arise when studying Interval Exchange Transformations. In 2003, Kontsevich and Zorich classified Extended Rauzy Classes by using data from Translation Surfaces, which are associated to IET's thanks to the Zippered Rectangle Construction of Veech from 1982. In 2009, Boissy finalized the classification of Rauzy Classes also using information from Translation Surfaces. We present in this paper specialized moves in (Extended) Rauzy Classes that allows us to prove the sufficiency and necessity in the previous classification theorems. These results provide a complete, and purely combinatorial, proof of these known results. We end with some general statements about our constructed move.\end{abstract}

\maketitle

\tableofcontents

\newpage

\section{Introduction}\label{SecIntro}

  Interval Exchange Transformations (or IET's) have been and continue to be a rich and interesting class of dynamical systems. In \cite{cVe1982}, W.\ A.\ Veech formalized the zippered rectangle construction, a method that permanently related IET's to Translation Surfaces, surfaces with atlases whose transition functions are given by translation away from a finite number of points.

  In \cite{cRau1979}, Rauzy introduced an induction on IETs. Later known as Rauzy Induction, this transforms an IET by taking its first return map onto a well chosen sub-interval. Rauzy Induction gave rise to Rauzy Classes and later Extended Rauzy Classes, equivalence classes of permutations. On zippered rectangles, Rauzy Induction extends to Rauzy-Veech Induction, which was realized to be a discretization of the Teichm\"uller Geodesic Flow on the Moduli Space of Abelian Differentials, a space associated with Translation Surfaces.

  It therefore became clear that connected components of the Moduli Space are in one to one correspondence with Extended Rauzy Classes. Using this, Extended Rauzy Classes were fully classified in \cite{cKonZor2003} using the following data from connected components:
      \begin{enumerate}
       \item the number and degrees of singularities of the surfaces, and
	\item one more description of the surfaces, namely:
	    \begin{enumerate}
	     \item whether or not the surfaces admit a hyperelliptic involution, and/or
	     \item the parity of the spin structure.
	    \end{enumerate}
      \end{enumerate}
  The two final descriptions only apply to special components. Therefore, many components are not hyperelliptic and do not have a parity value. For convenience, we shall call this final value a component's type (many will have type \textit{none}). So it follows that two permutations belong to the same Extended Rauzy Class if and only if their surfaces have the same type and degrees of singularities.

  In \cite{cBoi2012}, Boissy furthered these results to add one more piece of data, the degree of a \textit{marked singularity}. He concluded that two permutations belong to the same Rauzy Class if and only if these surfaces have the same type, degrees of singularities and degree of their marked singularities.

  This work is dedicated to proving the results mentioned above from \cite{cKonZor2003} and \cite{cBoi2012} while avoiding Translation Surfaces. To do this, we will develop combinatorial tools to analyze (Extended) Rauzy Classes. Specific combinations of Rauzy Induction, called \textit{switch moves}, will provide a relatively convenient way to transform permutations within (Extended) Rauzy Classes. Outside of Section \ref{SecIntro}, the author will make little mention of notions associated to surfaces (i.e. singularities) and instead rely on equivalent or alternate constructions.

\subsection{Outline of Paper}

In Section \ref{SecBackground}, we introduce our notation and terms for permutations, which we will call \emph{pairs}, and Rauzy Classes. We will define Rauzy Induction and Rauzy Paths. Included in this section, as Proposition \ref{PropStandardInClass}, is the classical result: every Rauzy Class contains a standard permutation. The function $\Sof{\pp}$ is a variant of a known map \cite[$\sigma$ in Section 2]{cVe1982}. A proof of its invariance within each Rauzy Class is provided. The map $\Sof{\pp}$ will yield derived objects $\Yof{\pp}$, $\Mof{\pp}$ and $\Pof{\pp}$. The degrees of singularities mentioned in the introduction will be replaced by $\Pof{\pp}$, while the degree of the marked singularity will be given by $\Mof{\pp}$. The value $\ARFof{\pp}$ will be used instead of the parity of spin structure mentioned in the introduction.

The local surgery called  ``bubbling a handle" from \cite{cKonZor2003} and its inverse operation are nicely defined combinatorially by adding or removing letters to permutations. We will formalize these ideas by defining prefix extensions and restrictions. This will give us the advantage to reduce longer permutations without losing or adversely affect desired properties.

The tools introduced in this paper are inner and outer \emph{switch moves}, simple paths of Rauzy Induction between standard permutations within Rauzy Class and Extended Rauzy Classes respectively. The proofs in Sections \ref{SecKZB} and \ref{SecOther} will deeply rely on these moves.

We will define piece-wise order reversing permutations. We will also further prepare for the proofs that follow by formalizing a decomposition of such permutations. In particular, Remark \ref{RemIgnoringChains} will be used freely in Section \ref{SecKZB}.

Section \ref{SecPWOR} is dedicated to proving the following theorem:
\begin{thmPWOR}
	Every Rauzy Class contains a piece-wise order reversing pair (i.e. permutation).
\end{thmPWOR}
\noindent This provides a purely combinatorial proof of \cite[Theorem 2.1]{cFick2014}, as all such permutations are equivalent to their own inverses.

Theorem \ref{ThmPWOrderReversing} will lead to further results in Section \ref{SecKZB}. Theorem \ref{ThmTypes} attributes the every Rauzy Class a Type, which aligns with the notion of type described in the introduction. Theorem \ref{ThmClassTypeSufficient} proves that Type and data from the invariant $\Sof{\pp}$ suffice to place two permutations in the same Rauzy Class. Theorem \ref{ThmExclassTypeSufficient} similarly shows that two permutations with the same Type and other data from $\Sof{\pp}$ belong to the same Extended Rauzy Class. Theorem \ref{ThmExclassUniqueType} completes the classification by showing that Type is a unique value to each Extended (and therefore regular) Rauzy Class. These results prove the classification theorems from \cite{cKonZor2003} and \cite{cBoi2012}.

Section \ref{SecSwitches} will prove two very important results concerning switch moves. The first result is that (inner) switch moves connect all standard permutations in a Rauzy Class. Likewise, the standard permutations in an Extended Rauzy Class are connected by (inner and outer) switch moves. These results are used in the proof of Theorem \ref{ThmExclassUniqueType}.

Section \ref{SecOther} provide other uses of the switch move. We revisit known results and provide simplified proofs by using switch moves.

The Appendices list ``utility" sequences of switch moves. These moves are provided in detail outside of the main proofs to aid in readability. They are divided into two portions, those that do not deal with $3$-blocks and those that do.

\subsection{Acknowledgments} The author would like to thank M. Kontsevich, A.\ Zorich and C.\ Boissy for their beautifully written results that form the focus of this work. He thanks D.\ Damanik for motivating Theorem \ref{ThmPWOrderReversing}, J.\ Chaika for his enduring encouragement on this and many other works and A.\ Sportiello for helpful comments on an earlier draft. The author would be unable to function without Amy and Charlie.

\section{Background and Definitions}\label{SecBackground}

  Throughout this paper, $\AAA$ will represent a finite alphabet and $\#\AAA$ its cardinality. We call a (possibly empty) ordered collection $\ol{w}$ of distinct letters in $\AAA$ a \term{word}. Given a word $\ol{w}$, let $\#\ol{w}$ denote its length.
\subsection{Rauzy Classes and Extended Rauzy Classes}

  Given finite alphabet $\AAA$, let $\perm{\AAA}$ denote the set of all pairs on $\AAA$ in the following way: $\pp\in\perm{\AAA}$ if $\pp = (p_0,p_1)$ where $p_\eps:\AAA\to\{1,\dots,\#\AAA\}$ are bijections, $\eps\in\{0,1\}$. We will often represent $\pp$ as
      $$ \pp = \mtrx{\LL{p_0^{-1}(1)}{p_1^{-1}(1)}~\LL{p_0^{-1}(2)}{p_1^{-1}(2)}~\LL{\dots}{\dots}~\LL{p_0^{-1}(\#\AAA)}{p_1^{-1}(\#\AAA)}}.$$
  As indicated by this representation, we will sometimes call $p_0$ or $p_1$ a \term{row}.

  \begin{remark}
      A pair $\pp=(p_0,p_1)$ is frequently called a ``permutation'' in literature about IET's that adopt this notation. As we will discuss later in this section, $\pp$ is related to a permutation and so this terminology is not without justification. However, we make the distinction here to be consistent with \cite{cFick2014b}.
  \end{remark}

  \begin{defn}\label{DefIrreducible}
      The set of \term{irreducible} pairs on $\AAA$ are
	$$ \irr{\AAA} := \left\{\pp\in\perm{\AAA}:~p_0^{-1}\{1,\dots,k\}=p_1^{-1}\{1,\dots,k\} \iff k=\#\AAA\right\}.$$
      In other words, $\pp\in\perm{\AAA}$ is irreducible if and only if it does not fix its first $k$ letters as a set for $k<\#\AAA$.  $\perm{\AAA}\setminus\irr{\AAA}$ is the set of \term{reducible} pairs.
  \end{defn}

  On $\irr{\AAA}$, we may define two maps. The \term{Rauzy Move} of Type $\eps$, $\eps\in\{0,1\}$, on $\pp=(p_0,p_1)\in\irr{\AAA}$ will be denoted by $\eps\pp$. If $\pp' = \eps\pp$ and $\pp' = (p'_0,p'_1)$ then $p_\eps' = p_\eps$ and
      $$ p'_{1-\eps}(b) = \RHScase{p_{1-\eps}(b), & p_{1-\eps}(b) \leq p_{1-\eps}(z),\\
				     p_{1-\eps}(b)+1, & p_{1-\eps}(z) < p_{1-\eps}(b) < \#\AAA,\\
				     p_{1-\eps}(z)+1 & p_{1-\eps}(b) = \#\AAA.}$$
  for each $b\in \AAA$, where $z = p_\eps^{-1}(\#\AAA)$. We note that for every $\pp\in\irr{\AAA}$ and $\eps\in\{0,1\}$, $\eps\pp\in\irr{\AAA}$ and there exists $k=k(\pp,\eps)>0$ such that $\eps^k\pp=\pp$ where $\eps^k\pp$ denotes acting on $\pp$ by a type $\eps$ Rauzy Move $k$ times. The following is therefore correctly named and defined.

  \begin{defn}\label{DefRauzyClass}
      Given $\pp\in\irr{\AAA}$, the \term{(labeled) Rauzy Class} of $\pp$, $\ClassLab{\pp}$, is the minimal equivalence class of pairs closed under both types of Rauzy Moves that contains $\pp$.
  \end{defn}

  We will call a sequence of Rauzy Moves a \term{Rauzy Path}. So an equivalent definition of a Rauzy Class could be given by the equivalence relation $\pp\sim\qq$ if and only if there exists a Rauzy Path connecting $\pp$ to $\qq$. When the staring pair is understood, we will express a Rauzy Path $\gamma$ by the types of Rauzy Moves in order. For example, $\gamma = 0^4 1^2 0^6$ would represent four moves of type $0$ followed by two moves of type $1$ and ending with six moves of type $0$. We call a path of the form $\eps^k$, $\eps\in\{0,1\}$ and $k>0$, a \term{cycle}. Given a path $\gamma = \eps_1^{k_1}\eps_2^{k_2}\cdots\eps_n^{k_n}$ such that $\eps_i\in\{0,1\}$ for $1\leq i \leq n$ and $\eps_{i+1} = 1-\eps_i$ for all $i<n$, let $\pp^{(0)},\pp^{(1)},\dots,\pp^{(n)}$ be defined as follows:
      \begin{equation}\label{EqCycleVertices}
	  \pp^{(0)} = \pp \mbox{ and } \pp^{(i)} = \eps_i^{k_i}\pp^{(i-1)}\mbox{ for } 1\leq i \leq n
      \end{equation}
  where $\gamma$ begins at $\pp$. We will call the pairs $\pp^{(i)}$ the \term{cycle vertices} of $\gamma$.

  A pair $\pp=(p_0,p_1)\in\irr{\AAA}$ is \term{standard} if the first letter of $p_0$ (resp. $p_1$) is the last letter of $p_1$ (resp. $p_0$), or equivalently $p_0^{-1}(1) = p_1^{-1}(\#\AAA)$ and $p_1^{-1}(1) = p_0^{-1}(\#\AAA)$. We denote by $\std{\AAA}$ the standard pairs on $\AAA$. The following result is very well known (see \cite{cRau1979}). However, we will use this proof later in Lemma \ref{LemDistanceBoundFromStandard}.

  \begin{prop}\label{PropStandardInClass}
    Every Rauzy Class contains a standard pair.
  \end{prop}

  \begin{proof}
      We shall start with any pair $\pp=(p_0,p_1)$ and connect it to a standard pair by Rauzy Moves. Note that if $\#\AAA=2$, the only irreducible pairs are standard. We shall therefore prove the proposition for $\#\AAA>2$. Let $z_\eps = p^{-1}_\eps(\#\AAA)$ and $n_\eps = p_{(1-\eps)}(z_\eps)$ for $\eps\in\{0,1\}$.

      If $\pp$ is not standard but $n_\eps = 1$ for some $\eps$, then we may apply $\#\AAA - p_{(1-\eps)}(a_\eps)$ Rauzy Moves of type $\eps$ to arrive at a standard pair, where $a_\eps = p^{-1}_\eps(1)$.

      Choose $\eps$ such that $1<n_\eps \leq n_{(1-\eps)} \leq \#\AAA$. Note that in fact $n_\eps \leq \#\AAA-2$ because $\pp$ is irreducible. Also, there must exist a letter $b$ such that
	  $$p_{(1-\eps)}(b) > n_\eps> p_\eps(b).$$
      By performing $\#\AAA - p_{(1-\eps)}(b)$ Rauzy Moves of type $\eps$ to $\pp$, we arrive at pair $\pp'$. This pair has the following new values:
	  $$ z_\eps' = z_\eps,~ n_\eps' = n_\eps,~ z_{(1-\eps)}' = b, n_{(1-\eps)}' = p_\eps(b) < n_\eps.$$

      We continue to replace $\pp$ with $\pp'$ by performing Rauzy Moves until $n_\eps = 1$ for some $\eps$. We then may arrive at a standard pair as mentioned above.
  \end{proof}

  We will call \term{Left Rauzy Moves} maps on $\pp\in\irr{\AAA}$ that cycle letters on the left of each row rather than the right. For $\eps\in\{0,1\}$, the Left Rauzy Move of type $\eps$ on $\pp=(p_0,p_1)$ will be denoted by $\tilde{\eps}\pp$. If $\pp' = \tilde{\eps}\pp$ and $\pp' = (p_0',p_1')$, then $p_\eps'=p_\eps$ while
      $$ p'_{1-\eps}(b) = \RHScase{p_{1-\eps}(a)-1 & p_{1-\eps}(b) = 1,\\
				     p_{1-\eps}(b)-1, & 1 < p_{1-\eps}(b) < p_{1-\eps}(a),\\
				     p_{1-\eps}(b), & p_{1-\eps}(b) \geq p_{1-\eps}(a).}$$
  for each $b\in \AAA$, where $a = p_\eps^{-1}(1)$. Just as with regular Rauzy Moves, these moves are each bijections on $\irr{\AAA}$. We are then lead to the following.

  \begin{defn}\label{DefExRauzyClass}
      The \term{(labeled) Extended Rauzy Class} of $\pp\in\irr{\AAA}$, denoted by $\ExtClassLab{\pp}$ is the minimal equivalence class of $\irr{\AAA}$ that contains $\pp$ and is closed under Rauzy Moves and Left Rauzy Moves of each type.
  \end{defn}

  We will now address (actual) permutations. A permutation on $\{1,\dots,n\}$ is simply a bijection from this set to itself, and the set of such permutations is denoted by $\mathfrak{S}_n$. There exists a map $\perm{\AAA}\to\mathfrak{S}_n$, $n=\#\AAA$, given by
      \begin{equation}\label{EqMonodromy}
      		(p_0,p_1) \mapsto p_1\circ p_0^{-1}.
      \end{equation}
  Irreducible permutations, Rauzy Moves and Left Rauzy Moves have definitions in $\mathfrak{S}_n$ that can be expressed through this map from $\perm{\AAA}$. The following is therefore acceptable for the purposes of this paper.

  \begin{defn}\label{DefOneRowClasses}
      Let $\ClassLab{\pp}$ (resp. $\ExtClassLab{\pp}$) be the Rauzy Class (resp. Extended Rauzy Class) of $\pp = (p_0,p_1)\in\irr{\AAA}$. The \term{non-labeled Rauzy Class} $\ClassNon{\pp}$ (resp. \term{non-labeled Extended Rauzy Class} $\ExtClassNon{\pp}$) associated to $\pp$ is the image of $\ClassLab{\pp}$ (resp. $\ExtClassLab{\pp}$) under the map from Equation \eqref{EqMonodromy}.
  \end{defn}

  The main results in Section \ref{SecKZB} refer to equivalent unlabeled classes and actual permutations. Two pairs $\pp=(p_0,p_1)\in\perm{\AAA}$ and $\pp'=(p'_0,p'_1)\in\perm{\AAA'}$ map to the same permutation if and only if they are the same pair under \term{renaming}, meaning $p_\eps' = p_\eps\circ\tau$, $\eps\in\{0,1\}$, for some bijection $\tau :\AAA' \to \AAA$. Note that if $\pp=(p_0,p_1)\in\perm{\AAA}$ maps to $\pi\in\mathfrak{S}_n$, then $(p_1,p_0)$ maps to $\pi^{-1}$. This notion of inverse will be used for $\pp\in\perm{\AAA}$. In particular if $\pp = (p_0,p_1)\in\irr{\AAA}$, then $\pp^{-1} = (p_1,p_0)\in\irr{\AAA}$ and
	  \begin{equation}\label{EqInverseMoves}
	      \eps\pp = \left[(1-\eps)\pp^{-1}\right]^{-1}
	  \end{equation}
  for $\eps\in\{0,1\}$. We will use this fact quite often, as it allows a proof one type of Rauzy Move to automatically apply to the other.

\subsection{Invariants For Rauzy Classes}

  This section will establish some class invariant values associated to pairs. These all are extensions of the known invariant $\sigma$ stated in Section 2 from \cite{cVe1982}. The invariants in this section are given for the convention of pairs in this work. Because of necessary differences in what follows from the classical map $\sigma$, all results will be proven.

  The map $\Sof{\pp}$ and the value $\Mof{\pp}$ given below are both invariants of a Rauzy Class, as shown in Proposition \ref{PropSofPiIsInvariant} and Corollary \ref{CorPofPiAndMofPiClassInvariant}. The map $\Yof{\pp}$, the list $\Pof{\pp}$ and the value $\ARFof{\pp}$ are invariants of an Extended Rauzy Class, as shown in Proposition \ref{PropPofPiExclassInvariant} and Corollary \ref{CorArfIsInvariant}. These will be used in Section \ref{SecKZB} as part of the classification of non-labeled Rauzy Classes and Extended Rauzy Classes.

  \begin{defn}\label{DefSofPiAndMofPi}
      Let $\pp=(p_0,p_1) \in \irr{\AAA}$. We define $\Sof{\pp} : \AAA \to \AAA$ by
	  $$ \Sof{\pp}(\alpha) = \RHScase{p_0^{-1}(1), & p_1(\alpha) = 1 \\
		p_0^{-1}(p_0(p_1^{-1}(\#\AAA))+1), & p_1(\alpha) = p_1(p_0^{-1}(\#\AAA))+1\\
		p_0^{-1}(p_0(p_1^{-1}(p_1(\alpha) -1))+1), & \mathrm{otherwise.}}$$
      Let $\Mof{\pp}$ denote the length of the cycle of $\Sof{\pp}$ that contains $p_0^{-1}(1)$.
  \end{defn}

  The following alternative definition of $\Sof{\pp}$ will be convenient for the proofs that follow.

  \begin{remark}\label{RemSofPiIsRL}
      For $\AAA$, $n:=\#\AAA$, let $\AAA^*:=\AAA\sqcup\{\lhd,\rhd\}$. Given $\pp = (p_0,p_1)$, define for $\eps\in\{0,1\}$ maps $\mathbf{L}_{p_\eps}, \mathbf{R}_{p_\eps}:\AAA^*\to\AAA^*$ by
	  $$ \mathbf{L}_{p_\eps}(b) = \RHScase{\rhd, & b=\lhd,\\
		      \lhd, & b\in\AAA,~ p_\eps(b)=1,\\
		      p_\eps^{-1}(p_\eps(b)-1), &b\in\AAA,~ p_\eps(b) >1,\\
		      p_\eps^{-1}(n), & b=\rhd.}$$
      for all $b\in\AAA^*$ and let $\mathbf{R}_{p_\eps} = \mathbf{L}_{p_\eps}^{-1}$. We may verify then that $\Sof{\pp}$ is the first return of $\mathbf{R}_{p_0}\circ\mathbf{L}_{p_1}$ to $\AAA$.
  \end{remark}

  Using the relationship between inverses and Rauzy Moves established in Equation \eqref{EqInverseMoves}, the following lemma will prove useful to the proof of Proposition \ref{PropSofPiIsInvariant}

  \begin{lem}\label{LemSofPiInverse}
   Let $\pp = (p_0,p_1)\in\irr{\AAA}$ and $\qq= (p_1,p_0)$ be its inverse. Then $\Sof{\qq} = \Sof{\pp}^{-1}$.
  \end{lem}
  \begin{proof}
      By Remark \ref{RemSofPiIsRL}, we see that $\Sof{\pp}$ is the first return of $\mathbf{R}_{p_0}\circ\mathbf{L}_{p_1}$ on $\AAA$, while $\Sof{\qq}$ is the first return map of $\mathbf{R}_{p_1}\circ\mathbf{L}_{p_0} = \big[\mathbf{R}_{p_0}\circ\mathbf{L}_{p_1}\big]^{-1}$ to $\AAA$. Because they are first returns of inverse functions, they are also inverses.
  \end{proof}

  \begin{prop}\label{PropSofPiIsInvariant}
      Let $\pp\in\irr{\AAA}$ and $\qq\in\ClassLab{\pp}$. Then $\Sof{\qq} = \Sof{\pp}$.
  \end{prop}

  \begin{proof}
      In light of Equation \eqref{EqInverseMoves} and Lemma \ref{LemSofPiInverse}, it suffices to prove the claim for $\qq = 1 \pp$. To ease in notation, let $a_i = p_0^{-1}(i)$ and $b_i = p_1^{-1}(i)$ for $1\leq i\leq n$ where $n=\#\AAA$. Also, let $j,k$ satisfy $a_j=b_n$ and $b_k = a_n$. We may assume that $\qq\neq \pp$ and so $a_{n-1} \neq b_n$. If $\qq = (q_0,q_1)$ and $\pp = (p_0,p_1)$, then $\mathbf{L}_{p_1} = \mathbf{L}_{q_1}$ and
	  $$ \mathbf{R}_{q_0}(c) = \RHScase{\mathbf{R}_{p_0}(c), & c\neq a_j,a_{n-1},a_n,\\
					  a_n, & c = a_j, \\
					  \rhd, & c = a_{n-1}\\
					  a_{j+1}, & c = a_n,}$$
    for each $c\in\AAA^*$ (see Remark \ref{RemSofPiIsRL}). We may then verify the claim by cases. Then
	  $$ \begin{array}{rclcl}
		  \Sof{\pp}(b_{k+1}) &=& \mathbf{R}_{p_0}\circ\mathbf{L}_{p_1}\circ\mathbf{R}_{p_0}\circ\mathbf{L}_{p_1}(b_{k+1})
			    & = & \mathbf{R}_{p_0}\circ\mathbf{L}_{p_1}\circ\mathbf{R}_{p_0}(a_n)\\
			    & = & \mathbf{R}_{p_0}\circ\mathbf{L}_{p_1}(\rhd)
			    ~ = ~ \mathbf{R}_{p_0}(a_j) &=& a_{j+1},
	     \end{array}$$
	  and
	  $$ \begin{array}{rclll}
		  \Sof{\qq}(b_{k+1}) &=& \mathbf{R}_{q_0}\circ\mathbf{L}_{q_1}(b_{k+1})& &\\
			    & = & \mathbf{R}_{q_0}(a_n) &=& a_{j+1}.
	     \end{array}$$
	Let $\ell$ satisfy $b_\ell = a_{n-1}$. Then
	   $$ \begin{array}{rclll}
		  \Sof{\pp}(b_{\ell+1}) &=& \mathbf{R}_{p_0}\circ\mathbf{L}_{p_1}(b_{\ell+1})& &\\
			    & = & \mathbf{R}_{p_0}(a_{n-1}) &=& a_n,
	     \end{array}$$
	and
	  $$ \begin{array}{rclcl}
		  \Sof{\qq}(b_{\ell+1}) &=& \mathbf{R}_{q_0}\circ\mathbf{L}_{q_1}\circ\mathbf{R}_{q_0}\circ\mathbf{L}_{q_1}(b_{\ell+1})\\
			    & = & \mathbf{R}_{q_0}\circ\mathbf{L}_{q_1}\circ\mathbf{R}_{q_0}(a_{n-1})\\
			    & = & \mathbf{R}_{q_0}\circ\mathbf{L}_{q_1}(\rhd)
			    ~ = ~ \mathbf{R}_{q_0}(a_j) &=& a_n.
	     \end{array}$$
	For all other letters, the claim is immediate.
  \end{proof}

  While Rauzy Moves fix the letters $p_0^{-1}(1)$ and $p_1^{-1}(1)$ for any (and therefore every) $\pp=(p_0,p_1)$ in a Rauzy Class, Left Rauzy Moves generally will not. It therefore becomes necessary to define a new object that will yield an invariant for Extended Rauzy Classes.

  \begin{defn}\label{DefTildeSofPiAndPofPi}
      For $\pp=(p_0,p_1)\in\irr{\AAA}$ and $a = p_0^{-1}(1)$, let $\Yof{\pp}$ be the first return of $\Sof{\pp}$ on the sub-alphabet $\AAA\setminus\{a\}$, or
	  $$ \Yof{\pp}(b) = \RHScase{\Sof{\pp}^2(b), & p_1(b)=1, \\ \Sof{\pp}(b), & \mathrm{otherwise}.}$$
      Let $\Pof{\pp}$ be the (unordered) list of lengths of cycles of $\Yof{\pp}$, with multiplicity.
  \end{defn}

  \begin{cor}\label{CorPofPiAndMofPiClassInvariant}
    Let $\ClassLab{\pp}\subset \irr{\AAA}$ be the Rauzy Class of $\pp\in\irr{\AAA}$. The value $\Mof{\qq}$ and the list $\Pof{\qq}$ are the same for every $\qq\in\ClassLab{\pp}$.
  \end{cor}

  \begin{proof}
   By Proposition \ref{PropSofPiIsInvariant}, $\Sof{\pp} = \Sof{\qq}$ for all $\qq\in\ClassLab{\pp}$. Because $\Yof{\pp}$, $\Pof{\pp}$ and $\Mof{\pp}$ are defined by $\Sof{\pp}$, the corollary follows.
  \end{proof}

  Just as Lemma \ref{LemSofPiInverse} aided in the proof of Proposition \ref{PropSofPiIsInvariant}, the following lemma will dramatically simplify the proof of Proposition \ref{PropPofPiExclassInvariant}, as Left Rauzy Moves and Rauzy Moves are related by the map $H$ defined below.

  \begin{lem}\label{LemHOfPi}
      Let $H:\perm{\AAA} \to \perm{\AAA}$ be the map defined by $H(\pp) = (p_0',p_1')$ where for each $b\in\AAA$ and $\eps\in\{0,1\}$,
	$$ p_\eps' (b) = \#\AAA+1-p_\eps(b).$$
      Then $\mathbf{L}_{p_0}\circ\Yof{\pp}\circ \mathbf{R}_{p_0} = \Yof{H(\pp)}^{-1}$ on $\AAA\setminus{p_0^{-1}(\#\AAA)}$ for $\pp=(p_0,p_1)\in\irr{\AAA}$.
  \end{lem}

  \begin{proof}
      Fix $\pp$ and let $H(\pp) = (p'_0,p_1')$. To aid in calculation, let $a_i = p_0^{-1}(i)$ and $b_i = p_1^{-1}(i)$ for $1\leq i \leq n$, where $n=\#\AAA$.
      We will verify that
		\begin{equation}\label{EqRSLS}
			\mathbf{R}_{p_0}\circ\Yof{H(\pp)}\circ\mathbf{L}_{p_0}\circ \Yof{\pp}(c) = c
		\end{equation}
		for each $c\in\AAA\setminus\{a_1\}$. Note that
		$$ \mathbf{L}_{p_\eps'}(c) = \RHScase{\mathbf{R}_{p_\eps}(c), & c\neq \lhd,\rhd,p_\eps^{-1}(n),\\
					\lhd, & c = p_\eps^{-1}(n),\\
					\rhd, & c = \lhd,\\
					p_\eps^{-1}(1), & c=\rhd,}$$
		and
		$$ \mathbf{R}_{p_\eps'}(c) = \RHScase{\mathbf{L}_{p_\eps}(c), & c\neq \lhd,\rhd,p_\eps^{-1}(1),\\
					\rhd, & c = p_\eps^{-1}(1),\\
					p_\eps^{-1}(n), & c = \lhd,\\
					\lhd, & c=\rhd,}$$
		for each $c\in\AAA^*$, referring to Remark \ref{RemSofPiIsRL}. We will verify the special cases of Equation \eqref{EqRSLS} momentarily, but note for ``typical" $c$,
			$$ \mathbf{R}_{p_0}\circ\Yof{H(\pp)}\circ\mathbf{L}_{\pi_0}\circ \Yof{\pp}(c) =
				\mathbf{R}_{p_0}\circ\mathbf{L}_{p_0}\circ\mathbf{R}_{p_1}\circ\mathbf{L}_{p_0}\circ \mathbf{R}_{p_0}\circ\mathbf{L}_{p_1}(c) = c.$$
		To deal with our special cases, let $j,k,\ell,r$ satisfy
			$$ a_j = b_1, ~ a_\ell = b_n, ~b_k=a_1,~\mbox{and}~b_r=a_n.$$
			
		\textit{Case 1:} $k=r+1$. We see that
				$$ \begin{array}{rclcl}
						\Yof{\pp}(a_j) & = & \mathbf{R}_{p_0}\circ\mathbf{L}_{p_1} \circ\mathbf{R}_{p_0}\circ\mathbf{L}_{p_1} \circ\mathbf{R}_{p_0}\circ\mathbf{L}_{p_1} (b_1)\\										&=& \mathbf{R}_{p_0}\circ\mathbf{L}_{p_1} \circ\mathbf{R}_{p_0}\circ\mathbf{L}_{p_1} \circ\mathbf{R}_{p_0}(\lhd)\\
										&=& \mathbf{R}_{p_0}\circ\mathbf{L}_{p_1} \circ\mathbf{R}_{p_0}\circ\mathbf{L}_{p_1}(b_k)\\
										&=& \mathbf{R}_{p_0}\circ\mathbf{L}_{p_1} \circ\mathbf{R}_{p_0}(a_n)
										~=~ \mathbf{R}_{p_0}\circ\mathbf{L}_{p_1}(\rhd)
										~=~ \mathbf{R}_{p_0}(a_{\ell}) &=& a_{\ell+1}.
						\end{array}$$
						and
				$$ \begin{array}{rclcl}
						\Yof{H(\pp)}(a_\ell) & = & \mathbf{R}_{p_0'}\circ\mathbf{L}_{p_1'} \circ\mathbf{R}_{p_0'}\circ\mathbf{L}_{p_1'} \circ\mathbf{R}_{p_0'}\circ\mathbf{L}_{p_1'} (b_n)\\
										&=& \mathbf{R}_{p_0'}\circ\mathbf{L}_{p_1'} \circ\mathbf{R}_{p_0'}\circ\mathbf{L}_{p_1'} \circ\mathbf{R}_{p_0'}(\lhd)\\
										&=& \mathbf{R}_{p_0'}\circ\mathbf{L}_{p_1'} \circ\mathbf{R}_{p_0'}\circ\mathbf{L}_{p_1'}(b_r)\\
										&=& \mathbf{R}_{p_0'}\circ\mathbf{L}_{p_1'} \circ\mathbf{R}_{p_0'}(a_1)
										~=~ \mathbf{R}_{p_0'}\circ\mathbf{L}_{p_1'}(\rhd)
										~=~ \mathbf{R}_{p_0'}(a_{j}) &=& a_{j-1}.
					\end{array}$$
		Therefore
			$$ \begin{array}{rclcl}
				\mathbf{R}_{p_0} \circ\Yof{H(\pp)}\circ\mathbf{L}_{p_0}\circ\Yof{\pp}(a_j) 
					&=& \mathbf{R}_{p_0} \circ\Yof{H(\pp)}\circ\mathbf{L}_{p_0}(a_{\ell+1}) \\
					&=& \mathbf{R}_{p_0} \circ\Yof{H(\pp)}(a_\ell)
					~=~ \mathbf{R}_{p_0}(a_{j-1})
					&=& a_j.\end{array}$$
		All other letters are ``typical" for this case.
			
		\textit{Case 2:} $k\neq r+1$. Let $s,t$ satisfy $$a_s = b_{k-1}~\mbox{and}~a_t = b_{r+1}.$$
		We find that
			$$ \begin{array}{rclcl}
						\Yof{\pp}(a_j) & = & \mathbf{R}_{p_0}\circ\mathbf{L}_{p_1} \circ\mathbf{R}_{p_0}\circ\mathbf{L}_{p_1}(b_1)
										&=& \mathbf{R}_{p_0}\circ\mathbf{L}_{p_1} \circ\mathbf{R}_{p_0}(\lhd)\\
										&=& \mathbf{R}_{p_0}\circ\mathbf{L}_{p_1}(b_k)
										~=~ \mathbf{R}_{p_0}(a_s) &=& a_{s+1},
						\end{array}$$
		and
			$$ \begin{array}{rclcl}
						\Yof{H(\pi)}(a_s) & = & \mathbf{R}_{\pi_0'}\circ\mathbf{L}_{p_1'} \circ\mathbf{R}_{p_0'}\circ\mathbf{L}_{p_1'}(b_{k-1})
										&=& \mathbf{R}_{p_0'}\circ\mathbf{L}_{p_1'} \circ\mathbf{R}_{p_0'}(a_1)\\
										&=& \mathbf{R}_{p_0'}\circ\mathbf{L}_{p_1'}(\rhd)
										~=~ \mathbf{R}_{p_0'}(a_j) &=& a_{j-1}.
						\end{array}$$
		Therefore
			$$ \begin{array}{rclcl}
				\mathbf{R}_{p_0} \circ\Yof{H(\pp)}\circ\mathbf{L}_{p_0}\circ\Yof{\pp}(a_j)
					&=& \mathbf{R}_{p_0} \circ\Yof{H(\pp)}\circ\mathbf{L}_{p_0}(a_{s+1})\\
					&=& \mathbf{R}_{p_0} \circ\Yof{H(\pp)}(a_s)
					~=~ \mathbf{R}_{p_0}(a_{j-1})
					&=& a_j.\end{array}$$
		We also find that
			$$ \begin{array}{rclcl}
						\Yof{\pp}(a_t) & = & \mathbf{R}_{p_0}\circ\mathbf{L}_{p_1} \circ\mathbf{R}_{p_0}\circ\mathbf{L}_{p_1}(b_{r+1})
										&=& \mathbf{R}_{p_0}\circ\mathbf{L}_{p_1} \circ\mathbf{R}_{p_0}(a_n)\\
										&=& \mathbf{R}_{p_0}\circ\mathbf{L}_{p_1}(\rhd)
										~=~ \mathbf{R}_{p_0}(a_\ell) &=& a_{\ell+1},
						\end{array}$$
		and
			$$ \begin{array}{rclcl}
						\Yof{H(\pp)}(a_\ell) & = & \mathbf{R}_{p_0'}\circ\mathbf{L}_{p_1'} \circ\mathbf{R}_{p_0'}\circ\mathbf{L}_{p_1'}(b_{n})
										&=& \mathbf{R}_{p_0'}\circ\mathbf{L}_{p_1'} \circ\mathbf{R}_{p_0'}(\lhd)\\
										&=& \mathbf{R}_{p_0'}\circ\mathbf{L}_{p_1'}(b_r)
										~=~ \mathbf{R}_{p_0'}(a_t) &=& a_{t-1}.
						\end{array}$$
		Therefore
			$$ \begin{array}{rclcl}
				\mathbf{R}_{p_0} \circ\Yof{H(\pp)}\circ\mathbf{L}_{p_0}\circ\Yof{\pp}(a_t) 
					&=& \mathbf{R}_{p_0} \circ\Yof{H(\pp)}\circ\mathbf{L}_{p_0}(a_{\ell+1})\\
					&=& \mathbf{R}_{p_0} \circ\Yof{H(\pp)}(a_\ell)
					~=~ \mathbf{R}_{p_0}(a_{t-1})
					&=& a_t.\end{array}$$
		All other letters are ``typical" in this case.
  \end{proof}

  \begin{prop}\label{PropPofPiExclassInvariant}
    Let $\ExtClassLab{\pp}\subset \irr{\AAA}$ be the Extended Rauzy Class for $\pp\in\irr{\AAA}$. The lists $\Pof{\pp}$ and $\Pof{\qq}$ for every $\qq\in\ExtClassLab{\pp}$.
  \end{prop}
  
	\begin{proof}
		By Corollary \ref{CorPofPiAndMofPiClassInvariant}, we must only show that $\Pof{\pp} = \Pof{\tilde{\eps}\pp}$ for $\eps\in\{0,1\}$.
		We may verify that $\tilde\eps \pp = H\big(\eps H(\pp)\big)$. By Proposition \ref{PropSofPiIsInvariant} and Lemma \ref{LemHOfPi}, it follows that $\Yof{\pp} = \Yof{\tilde{\eps}\pp}$, where $H$ is defined in Lemma \ref{LemHOfPi}. This concludes the proof, as $\Pof{\pp}$ and $\Pof{\tilde{\eps}\pp}$ are defined by $\Yof{\pp}$ and $\Yof{\tilde{\eps}\pp}$.
	\end{proof}

    For alphabet $\AAA$, let $\PP(\AAA)$ denote the set of unordered pairs of distinct letters in $\AAA$.

   \begin{defn}\label{DefCanonQuad}
      Given $\pp=(p_0,p_1)\in\irr{\AAA}$, the \term{canonical quadratic form} associated to $\pp$ will be $\QF{\pp}:\ZZ_2^\AAA \to \ZZ_2$ given by
	  $$ \QF{\pp}(v) = \sum_{a\in\AAA} v_a^2 + \sum_{\{a,b\}\in\PP(\AAA)} L_\pp(a,b) v_a v_b\mod 2$$
      where
	  $$ L_\pp(a,b) = \RHScase{1, &\big(p_0(a)-p_0(b)\big)\big(p_1(a)-p_1(b)\big) < 0,\\ 0, & \mathrm{otherwise}.}$$
  \end{defn}

    By definition, $L_\pp(a,b) = L_\pp(b,a)$ for all $a,b\in \AAA$. Therefore the definition of $\QF{\pp}$ as a sum of pairs $\PP(\AAA)$ is well-defined.

    \begin{defn}
	Two quadratic forms $\QQQ:\ZZ_2^\AAA\to \ZZ_2$ and $\QQQ':\ZZ_2^{\AAA'}\to \ZZ_2$, where $\AAA$ and $\AAA'$ are alphabets, are \term{equivalent} if there exists an invertible linear map $A:\ZZ_2^\AAA \to \ZZ_2^{\AAA'}$ so that $ \QQQ'\circ A = \QQQ$.
    \end{defn}

    \begin{prop}\label{PropQuadFormInvariance}
	Suppose $\qq\in \ExtClassLab{\pp}$ for $\pp\in\irr{\AAA$}, then $\QF{\pp}$ and $\QF{\qq}$ are equivalent. Furthermore, if $\qq\in \ExtClassNon{\pp}$, then $\QF{\pp}$ is equivalent to $\QF{\qq}$.
    \end{prop}

    \begin{proof}
	It suffices to show the first claim for only one inductive move, of type $0$ or $1$ on either the left or the right. If we recall the meaning of $\pp^{-1}$ as $(p_0,p_1)$, where $\pp = (p_0,p_1)$, and $H(\pp)$ as defined in Lemma \ref{LemHOfPi}, then $L_\pp = L_{\pp^{-1}} = L_{H(\pp)}$ as given in Definition \ref{DefCanonQuad}. Therefore by the relationships $1 \pp = \big( 0(\pp^{-1})\big)^{-1}$ and $\tilde{\eps}\pp = H(\eps H(\pp))$ for $\eps\in\{0,1\}$, we must only verify the case for $\qq = 0\pp$.

	Let $z_\eps$ satisfy $p_\eps(z_\eps)=\#\AAA$ for $\eps\in\{0,1\}$.
	We may verify directly that $$ L_{\qq}(c,z_1) = L_\pp(c,z_1) + L_\pp(z_0,z_1)\mbox{ for } c\in \AAA\setminus\{z_1\}$$ and $L_{\qq}= L_{\pp}$ otherwise. Let $A:\ZZ_2^\AAA\to \ZZ_2^\AAA$ be the invertible linear transformation given by
		$$ (Av)_b = \RHScase{ v_{z_0} + v_{z_1}, & b=z_0,\\ v_b, & \mathrm{otherwise}.}$$
	Then we may see that
		$$ \begin{array}{rcl}
			\QF{\pp}(Av) & = & v_{z_0}^2 + v_{z_1}^2 + \sum_{b\in\AAA\setminus\{z_0\}} v_b^2\\
			  & & + \sum_{\{b,z_0\}\in\PP_2(\AAA)} L_\pp(b,z_0)\cdot v_b(v_{z_0}+v_{z_1})\\
			  & & + \sum_{\{b,z_1\} \in \PP_2(\AAA\setminus\{z_0\})} L_\pp(b,z_1)\cdot v_b v_{z_1}\\
			  & & + \sum_{\{b,c\}\in\PP_2(\AAA\setminus\{z_0,z_1\})} L_\pp(b,c)\cdot v_b v_c\vspace{.05in}\\
			  & = & \sum_{b\in\AAA} v_b^2 + v_{z_1}^2 + L_\pp(z_0,z_1)v_{z_1}^2 \\
			  & & + \sum_{\{b,z_0\}\in\PP_2(\AAA)} L_\pp(b,z_0) v_b v_{z_0}\\
			  & & + \sum_{\{b,z_1\}\in\PP_2(\AAA\setminus\{z_0\})} (L_\pp(b,z_1) + L_\pp(b,z_0)) v_b v_{z_1}\\
			  & & + \sum_{\{b,c\}\in\PP_2(\AAA\setminus\{z_0,z_1\})} L_\pp(b,c) v_b v_c\vspace{.05in}\\
			  & = & \sum_{b\in\AAA} v_b^2 + \sum_{\{b,c\}\in\PP_2(\AAA)} L_{\qq}(\{b,c\}) v_b v_c\\
			  & = & \QF{\qq} (v).
		   \end{array}$$

	The second claim follows from the fact that $\pp=(p_0,p_1)\in \irr{\AAA}$ and $\qq\in \ExtClassNon{\pp}$ then there exists a bijection $\tau:\AAA \to \AAA'$ and $\pp'=(p'_0,p'_1) \in \ExtClassLab{\pp}$ so that $q_\eps = p'_\eps\circ\tau$, $\eps\in\{0,1\}$. Quadratic forms $\QF{\pp}$ and $\QF{\pp'}$ are equivalent by the first claim, and $\QF{\pp'}$ is equivalent to $\QF{\qq}$ by the natural linear isomorphism
		$ e_b \mapsto e_{\tau(b)}$
	generated by $\tau$.
    \end{proof}

    \begin{defn}
	If $\pp\in \irr{\AAA}$ with associated quadratic form $\QF{\pp}$, then
	    $$ \ARFof{\pp} := \#\{v\in \ZZ_2^\AAA: \QF{\pp}(v)=1 \mod 2\}.$$
    \end{defn}

    \begin{remark}
	The value $\ARFof{\pp}$ is related to the parity of the spin structure, a $\ZZ_2$ value associated to some pairs. See \cite{cDele2013} for an excellent exposition on how to define the parity using quadratic forms for appropriate $\pp$.
    \end{remark}

    \begin{cor}\label{CorArfIsInvariant}
	If $\pp\in\irr{\AAA}$ and $\qq \in \ExtClassNon{\pp}$, then $\ARFof{\qq}=\ARFof{\pp}$.
    \end{cor}

\subsection{Adding and Removing Letters}

  The following concepts have interpretations concerning Translation Surfaces (see \cite{cKonZor2003} and their ``bubbling a handle'' operation as a comparison). However, we will be using these concepts for their combinatoric properties alone. In the remaining sections, these results will allow us to ignore all but specific letters when altering a pair.

  \begin{defn}\label{DefExtensions}
    Let $\AAA'\subsetneq \AAA$ be finite alphabets. A choice of words $$\Omega = \{\ol{\omega}_{\eps,b}\}_{\eps\in\{0,1\},b\in\AAA'}$$ such that
      \begin{itemize}
       \item For each $\eps\in\{0,1\}$ and $b\in\AAA'$, $\ol{\omega}_{\eps,b} = \ol{u} b$ where $\ol{u}$ is a (possibly empty) word in $\AAA\setminus\AAA'$.
	\item For each $\eps$ and $c\in \AAA$, there exists $b\in \AAA'$ so that $c$ appears in $\ol{\omega}_{\eps,b}$.
	\item For each $\eps$ and $b\neq c$, $\ol{\omega}_{\eps,b}$ and $\ol{\omega}_{\eps,c}$ do not share a common letter, i.e. there does not exist $d$ that appears in both words.
      \end{itemize}
    will be called a \term{(prefix) extension} from $\AAA'$ to $\AAA$. If $\pp=(p_0,p_1)\in\perm{\AAA'}$, the \term{(prefix) extension} of $\pp$, $\Omega(\pp) = (q_0,q_1)\in \perm{\AAA}$, is given by
	$$ q_\eps(b) = \RHScase{\sum_{c\in\AAA': p_\eps(c) \leq p_\eps(b)} \#\ol{\omega}_{\eps,c}, & b \in \AAA'\\
			  k+\sum_{c\in\AAA': p_\eps(c) < p_\eps(d)} \#\ol{\omega}_{\eps,c},& b\notin\AAA',~ b\mbox{ is in position }k\mbox{ in }\ol{\omega}_{\eps,d}}$$
    for $\eps\in\{0,1\},b\in\AAA$.
  \end{defn}

  Suppose $\AAA'\subsetneq\AAA$ with extension $\Omega$ and $\pp=(p_0,p_1) \in \irr{\AAA'}$. Let $z_\eps = p_\eps^{-1}(\#\AAA')$ for $\eps\in\{0,1\}$. By fixing $\eps$, let $\pp'=\eps \pp$. If $\qq = \Omega(\pp)$ and $\qq' = \Omega(\pp')$, then it follows by definition that 
    \begin{equation}\label{EqExtendedMove}
    	\qq' = \eps^k \qq, \mbox{ where }k=\#\ol{\omega}_{1-\eps,z_{1-\eps}}
    \end{equation}
  By considering either a Rauzy Path between $\pp$ and $\pp'$, we may find a path from $\Omega(\pp)$ to $\Omega(\pp')$ using Equation \eqref{EqExtendedMove} on each move.

  \begin{defn}\label{DefPathExtensions}
    We call a path $\gamma$ from $\Omega(\pp)$ to $\Omega(\pp')$ in $\AAA$ mentioned in the previous paragraph an \term{extension} of a path $\gamma'$ from $\pp$ to $\pp'$ in $\AAA'$.
  \end{defn}

  The prefix extension maps from pairs on one alphabet $\AAA'$ to a larger alphabet $\AAA$. At least in one direction, this operation has a natural inverse.

  \begin{defn}\label{DefRestrictions}
      Let $\pp=(p_0,p_1)\in\perm{\AAA}$ and $\AAA'\subset \AAA$ such that $p_\eps^{-1}(\#\AAA)\in\AAA'$. The \term{(prefix) restriction} of $\pp$ to $\AAA'$ is denoted by $\pp\RED{\AAA'}$. If $\pp\RED{\AAA'}= (\pp'_0,\pp'_1)$, then
	$$ p'_\eps(b) = \#\big\{c\in\AAA': p_\eps(c)\leq p_\eps(b)\big\}$$
      for each $b\in\AAA'$ and $\eps\in\{0,1\}$.
  \end{defn}

  Extensions and restrictions are in essence inverse operations in the following way: If $\Omega:\AAA'\to\AAA$ is an extension, then $\Omega(\pp)\RED{\AAA'}=\pp$ for all $\pp\in\perm{\AAA'}$. If instead $\pp\in\perm{\AAA}$ and $\AAA'$ is an appropriate sub-alphabet, then there exists an extension $\Omega:\AAA'\to\AAA$ (depending on $\pi$) such that $\Omega\left(\pp\RED{\AAA'}\right)=\pp$. We will often define extensions implicitly by fixing a restriction.

  We may not simply consider prefix extensions and reductions as maps between $\irr{\AAA}$ and $\irr{\AAA'}$ as there are several examples of irreducible pairs whose extensions or restrictions are reducible. We first will give a necessary and sufficient condition for an extension to preserve irreducibility.

  \begin{lem}\label{LemIrreducibleExtensions}
      Let $\pp =(p_0,p_1)\in \irr{\AAA'}$ and $a_\eps = p^{-1}_\eps(1)$ for $\eps\in\{0,1\}$. Consider any prefix extension $\Omega:\perm{\AAA'} \to \perm{\AAA}$.
      
      Then $\Omega(\pp)\in\irr{\AAA}$ if and only if $\ol{\omega}_{0,a_0}$ and $\ol{\omega}_{1,a_1}$ do not share the same first $k$ letters, for any $k>0$.
  \end{lem}

  \begin{proof}
      Let $\Omega(\pp) = (q_0,q_1)$. Suppose first that $\ol{\omega}_{0,a_0} = \ol{u}_0\ol{v}_0$ and $\ol{\omega}_{1,a_1} = \ol{u}_1\ol{v}_1$ where $\ol{u}_0$ and $\ol{u}_1$ are both words of length $k>0$ that are composed the letters in $\BBB\subseteq\AAA\setminus\AAA'$, $\#\BBB=k$. Then $q_0(\BBB) = q_1(\BBB) = \{1,\dots,k\}$ or $\Omega(\pp)$ is reducible.

      Now suppose that $\Omega(\pp)=(q_0,q_1)$ is reducible. Then there is a nontrivial partition $\AAA = \AAA_-\sqcup \AAA_+$ such that
	  $$ q_0(\AAA_-) = q_1(\AAA_-) = \{1,\dots,\#\AAA_-\}.$$
      We claim that $\AAA_-\cap \AAA' = \emptyset$. This would then imply that $\AAA_-$ is the set of letters that begin both $\ol{\omega}_{0,a_0}$ and $\ol{\omega}_{1,a_1}$, concluding our proof. To show the claim, note that $\AAA_-\cap\AAA'$ would be the first letters in each row of $\pp$. Because $\pp$ is irreducible, this must be empty.
  \end{proof}

  We shall now give a necessary and sufficient condition to ensure that a Rauzy Path in $\AAA$ is an extension of a Rauzy Path $\AAA'$, given that our initial pair in $\AAA$ and its restriction to $\AAA'$ are irreducible.

  \begin{lem}\label{LemReductionPaths}
      Let $\pp=(p_0,p_1)\in\irr{\AAA}$ and let $\AAA'\subset \AAA$ be such that $\pp\RED{\AAA'} \in \irr{\AAA'}$ and $p^{-1}_\eps\in\AAA'$ for $\eps\in\{0,1\}$. For any path $\gamma$ from $\pp$ to $\qq$, the following are equivalent:
	  \begin{itemize}
	   \item[(i)] $\gamma$ is an extension of a path $\gamma'$ from $\pp\RED{\AAA'}$ to $\qq\RED{\AAA'}$.
	    \item[(ii)]  If $\pp = \pp^{(0)}, \pp^{(1)},\dots,\pp^{(k)}=\qq$ are the cycle vertices of $\gamma$ and
		  $$ \BBB(\gamma) := \{b\in\AAA: p^{(i)}_\eps(b) = \#\AAA\mbox{ for some }\eps\in\{0,1\},0\leq i \leq k\}$$
		denote the last letters in each cycle vertex, then $\BBB(\gamma) \subseteq \AAA'$.
	  \end{itemize}
  \end{lem}

  \begin{proof}
    Suppose $\gamma$ is an extension from $\gamma'$ in $\AAA'$. Then $\gamma'$ is a path with vertices $$\pp^{(0)}\RED{\AAA'},\dots, \pp^{(k)}\RED{\AAA'}$$ and $\BBB(\gamma')\subseteq\AAA'$. For $0\leq i \leq k$, $\pp^{(i)}$ and $\pp^{(i)}\RED{\AAA'}$ coincide on the last letter in each row. Therefore $\BBB(\gamma) = \BBB(\gamma') \subseteq\AAA'$.

    Now suppose $\BBB(\gamma)\subseteq\AAA'$. It follows that for each $0\leq i < k$, the cycle connecting $\pp^{(i)}$ to $\pp^{(i+1)}$ is an extension of the cycle connecting $\pp^{(i)}\RED{\AAA'}$ to $\pp^{(i+1)}\RED{\AAA'}$. Because this is true for each cycle, the entire path $\gamma$ is therefore an extension of a path in $\AAA'$.
  \end{proof}

  In the proof of Proposition \ref{PropSwitchesClass}, we will in the general case assume that a pair $\pp$ admits an irreducible restriction when removing at least two letters. This is not always so. The following class of pairs, $\STAR{\AAA}$, will be shown in Proposition \ref{LemFarthestFromStandard} to not admit any irreducible restriction. By identifying these as the exhaustive set of such pairs, we will be able to prove Proposition \ref{PropSwitchesClass} with only a few special cases. We will begin by defining a notion of distance in Rauzy Classes.

  Given a Rauzy Path $\gamma$ between two pairs $\pp$ and $\pp'$, we may define the \term{cycle length} of $\gamma$ to be number of cycles in the path. In other words
      \begin{equation}\label{EqCycleLength}
	  \Ncycle \gamma = n,\mbox{ for }\gamma = \eps_1^{k_1}\cdots\eps_n^{k_n}
      \end{equation}
    where $\eps_{i+1} = 1-\eps_i$ for $i<n$ and $k_i>0$. Note that if $\gamma$ is a path from $\pp$ to $\pp'$, there exists a path $\gamma'$ from $\pp'$ to $\pp$ such that $\Ncycle\gamma = \Ncycle\gamma'$. With this, we may naturally define a metric on each Rauzy Class as follows
      \begin{equation}\label{EqDCycle}
	  \dcycle(\pp,\pp') = \RHScase{\displaystyle \min_{\pp \overset{\gamma}\rightarrow \pp'}\vspace{.04in} \Ncycle\gamma, & \pp \neq \pp',\\
	  					0, & \pp = \pp',}
      \end{equation}
      where the minimum is taken over all paths $\gamma$ from $\pp$ to $\pp'\neq \pp'$.
      We will now use our proof of well-known result Proposition \ref{PropStandardInClass} to yield a distance bound between any irreducible pair and a standard pair.

  \begin{lem}\label{LemDistanceBoundFromStandard}
      For all $\pp\in\irr{\AAA}$, $\dcycle(\pp,\std{\AAA}) \leq \#\AAA-2$.
  \end{lem}

  \begin{proof}
      By examining the proof of Proposition \ref{PropStandardInClass}, we see that any pair may be connected to a standard pair by a cycle path of length at most $\#\AAA-2$. The maximum value of $n_\eps$ must be $\#\AAA-2$ initially. Each cycle reduces our value of $n_\eps$ by at least one. So after at most $\#\AAA-3$ cycles, $n_\eps=1$ for some $\eps\in\{0,1\}$. One more cycle will produce a standard pair.
  \end{proof}

  We now define a special set of pairs. These will be shown to be the set of pairs that achieve the maximum distance from standard pairs. 

  \begin{defn}\label{DefSpecialForm}
    For finite alphabet $\AAA$, $\#\AAA\geq 2$, denote by $\STAR{\AAA}\subseteq\irr{\AAA}$ the set of pairs $\pp = (p_0,p_1)$ such that:
	\begin{itemize}
	 \item There exist $a,b\in\AAA$ and $\eps,\eps'\in\{0,1\}$ such that
		$$p_{\eps'}(a) = 1,~p_{1-\eps'}(a) = 2,~p_{\eps}(b) = \#\AAA,~p_{1-\eps}(b)=\#\AAA-1.$$
	 \item For all $c\in\AAA\setminus\{a,b\}$, $ |p_0(c)-p_1(c)|=2.$
	\end{itemize}
    For $\pp\in\STAR{\AAA}$, let $\eps(\pp)$ and $\eps'(\pp)$ be the values $\eps$ and $\eps'$ respectively as stated above.
  \end{defn}

  \begin{remark}
    By inspection of this definition, we see that we may equivalently say that $\pp\in\STAR{\AAA}$ takes on the following forms. Let $n=\#\AAA$ and for each case below, let $a_1,\dots,a_n\in\AAA$ be defined to satisfy $p_0(a_j)=j$ for all $1\leq j \leq n$.

    If $n$ is even and $\eps(\pp)=0$,
	  $$p_1(a_j) = \RHScase{2, & j=1 \\ j+2, & 1<j<n\mbox{ is even}\\ j-2, & 1<j<n\mbox{ is odd}\\ n-1, &j=n.}$$

    If $n$ is even and $\eps(\pp)=1$,
	  $$p_1(a_j) = \RHScase{1, & j=2 \\ j+2, & 1\leq j<n-1\mbox{ is odd}\\ j-2, & 1<j\leq n \mbox{ is even}\\ n, &j=n-1.} $$

    If $n$ is odd and $\eps(\pp)=0$,
	  $$p_1(a_j) = \RHScase{1, & j=2 \\ j+2, & 1\leq j<n\mbox{ is odd}\\ j-2, & 2<j<n\mbox{ is even}\\ n-1, &j=n.}$$

    If $n$ is odd and $\eps(\pp)=1$,
	  $$p_1(a_j) = \RHScase{2, & j=1 \\ j+2, & 2\leq j<n-1 \mbox{ is even}\\ j-2, & 1<j\leq n \mbox{ is odd}\\ n, &j=n-1.} $$

    Note in particular that $\eps'(\pp)$ depends on $n$ and $\eps(\pp)$ in that $n+\eps(\pp)+\eps'(\pp)$ must be even. We will therefore only refer to $\eps(\pp)$ when discussing such a pair.
  \end{remark}

  The following result relates the pairs in $\STAR{\AAA}$ with the distance bound given in Lemma \ref{LemDistanceBoundFromStandard}. It also relates the notion of distance with how many letters may be removed while keeping a pair irreducible.

  \begin{lem}\label{LemFarthestFromStandard}
      For $\pp=(p_0,p_1)\in\irr{\AAA}$, the following are equivalent:
      \begin{itemize}
       \item[(i)] $\dcycle(\pp,\std{\AAA}) = \#\AAA-2.$
	\item[(ii)] for all $b\in\AAA\setminus\{p_0^{-1}(1),p_1^{-1}(1),p_0^{-1}(\#\AAA),p_1^{-1}(\#\AAA)\}$, $\pp\RED{\AAA\setminus\{b\}}$ is reducible.
	\item[(iii)] $\pp\in\STAR{\AAA}$.
      \end{itemize}
  \end{lem}

  \begin{proof}
	We first show that (i) implies (ii). Suppose $\pp\RED{\AAA'}$ is irreducible for $\AAA':=\AAA\setminus\{b\}$, $b\in\AAA$. Then there would exist a shortest path connecting $\pp$ to standard pair $\qq$ that would be an extension of a path in $\AAA'$. This path would necessarily have length at most $\#\AAA'-2 = \#\AAA-3$ by Lemma \ref{LemDistanceBoundFromStandard}.

	To show that (ii) implies (i), suppose that $\pp$ is connected to standard $\qq$ by a path of length $\ell < \#\AAA-2$. Then the connecting path is an extension on some alphabet $\AAA'\subsetneq\AAA$ where $\#\AAA' = \ell+2 < \#\AAA$. The pair $\pp\RED{\AAA'}$ is standard because $p^{-1}_0(1),p_1^{-1}(1),p^{-1}_0(\#\AAA),p_1^{-1}(\#\AAA)\in\AAA'$. Therefore $\pp\RED{\AAA'}$ must be irreducible.

	Now we will show that (i) implies (iii). We show this by induction on $\#\AAA$, noting that the case $\#\AAA=2$ is trivial and the case $\#\AAA=3$ is direct. By looking at the proof of Proposition \ref{PropStandardInClass}, if the cycle distance between $\pp$ and $\std{\AAA}$ is $\#\AAA-2$, then up to switching rows, $\pp$ must have the following form
	    $$ \pp = \mtrx{\LL{\dots}{\dots}~\LL{d_2~d_1~d_0}{d_0~d_i~d_1}}$$
	where $d_0,\dots,d_{\#\AAA-1}$ for the row $p_0$ in reverse order and $i>1$. Because $1\pp = \pp$, any path from $\pp$ to a standard pair must begin with a single $0$ move as its first cycle. Let $\pp' = 0\pp$ and $\AAA' := \AAA\setminus\{d_1\}$. We may verify that $\pp'\RED{\AAA'}$ is irreducible, so path must exist in $\AAA'$ that extends to a path in $\AAA$ which connects $\pi'$ to a standard pair. By our assumption of the distance on $\pp$, this extended path must have cycle length at least $\#\AAA-3$ in $\AAA'$. By induction, $\pp'$ must belong to $\STAR{\AAA'}$ with $\eps(\pp') = 0$ (see Definition \ref{DefSpecialForm}). It follows that $\pp\in\STAR{\AAA}$ with $\eps(\pp)=1$.

	We will now show that (iii) implies (ii). Choose $b\in\AAA$ that is not the first or last letter of either row of $\pp$. Let $\nu \in \{0,1\}$ satisfy $p_\nu(b) = p_{(1-\nu)}(b)+2$. Let $c$ be the letter such that $p_\nu(c) = p_\nu(b)-1$. Because $\pp\in\STAR{\AAA}$, $p_{(1-\nu)}(c) = p_\nu(c)+2 = p_\nu(b)+1$. Partition $\AAA$ by $\AAA = \AAA_{-} \sqcup \{b,c\} \sqcup \AAA_{+}$ where
	    $$ \AAA_{-} = \{d\in\AAA\setminus\{c\}: p_\nu(d) < p_\nu(b)\} \mbox{ and } \AAA_{+} = \{d\in\AAA: p_\nu(d) > p_\nu(b) \}.$$
	Note that there does exist $d\in\AAA_+$ such that $p_{(1-\nu)}(d) = p_\nu(b)$. It follows that if $\pp' = \pp\RED{\AAA\setminus\{b\}}$,
	    $ p_0'(\AAA_-) = \{1,\dots, \#\AAA_-\} = p'_1(\AAA_-)$
	or $\pp'$ is reducible.
	
  \end{proof}

  We conclude with one more result concerning $\STAR{\AAA}$. In particular we show a specific standard pair connected to such a pair.

  \begin{lem}\label{LemFarthestToHyp}
      Let $\pp=(p_0,p_1)\in\STAR{\AAA}$ and $a_1,a_2,\dots,a_n$ denote the letters of row $p_0$ in order, or equivalently $p_0(a_i) = i$ for $1\leq i \leq n = \#\AAA \geq 3$. Then there exists a path from $\pp$ to pair $\qq=(q_0,q_1)$ whose form is described by what follows:
	  \begin{itemize}
	   \item If $n + \eps(\pp)$ is even,
		  $$ q_0(a_j) = \RHScase{1, & j=1, \\ k+1, & j=2k, \\ n - k, & j=2k+3.}$$
	   \item If $n+ \eps(\pp)$ is odd,
		  $$ q_0(a_j) = \RHScase{k+1, & j=2k+1, \\ n +1 - k, & j= 2k.}$$
	   \item For each $1\leq j \leq n$, $$q_0(a_j) + q_1(a_j) = n+1.$$
	  \end{itemize}
  \end{lem}

  \begin{proof}
      We shall prove this by induction on $n=\#\AAA$. The case for $n = 3$ is direct to verify. So consider $n>3$. We will show one case, as the other cases all follow a similar argument. Suppose that $n$ is odd and $\eps(\pp) = 0$. We note that $0\pp=\pp$, so let $\pp' = 1\pp$. In this case, $p_1'=p_1$. Note also that $p'_0(a_n) = n-1$ and $p_0'(a_{n-1}) = n$. If we let $\#\AAA' := \AAA\setminus\{a_n\}$, we see that $\pp'\RED{\AAA'}$ is an element of $\STAR{\AAA'}$ with $\eps(\pp') = 1$.
      By induction, there is a path connecting $\pp'$ to $\qq$ that is the extension of a path from $\pp'\RED{\AAA'}$ to $\qq' = (q_0',q_1')$ where
	  $$ q'_0(a_j) = \RHScase{k+1, & j=2k+1, \\ n +1 - k, & j= 2k,}$$
      and $q'_0(a_j) + q'_1(a_j) = n$ for all $1\leq j \leq n-1$. By noting that $a_n$ was before $a_{n-1}$ in the first row and before $a_{n-2}$ in the second in our prefix extension, $\qq$ is indeed the desired form.
  \end{proof}

\subsection{The Switch Move and Extended Switch Move}

	By Proposition \ref{PropStandardInClass}, we know that standard pairs exist in each Rauzy Class. We explore in this section \term{switch moves}, specific Rauzy Paths that connect standard pairs to each other. These moves will be fundamental tools used in this paper.
	
	Consider  $\pp=(p_0,p_1)\in\std{\AAA}$ with letters $z_0 ,z_1\in\AAA$ such that $p_0(z_0)= p_1(z_1)=\#\AAA$. We begin with a simple question, what is the shortest cycle-length path from $\pp$ to another standard pair? We know that the length can not be $1$, as the resulting pair would no longer be standard. Up to switching rows, let $\pp'=0^{k_1}\pp$ be the first cycle vertex from $\pp$. If $z'_0$ and $z_1'$ are defined analogously for $\pp'$, then $z_0'=z_0$ and $z_1'=b$ for some letter $b\neq z_1$. Because $z_1'\neq z_1$, the next cycle vertex $\pp'' = 1^{k_2}\pp'$ cannot be standard either. It follows that $z_1''=b$ and $z_0''=c$ for some new letter $c$. Note that $p_0(c)>p_0(b)$. The reader may verify that if $p_1(c)>p_1(b)$, we may now choose $k_3$ such that $\pp''' = 0^{k_3}\pp''$ satisfies $z_0''' = c$ and $z_1'''=z_1$. It is then possible to choose $k_4$ such that $\pp^{(4)} = 1^{k_4}\pp'''$ is standard. This paragraph may be phrased to in fact prove that two distinct standard pairs in the same Rauzy Class must be separated by a cycle distance of at least four, and $c$ appears after $b$ in each row of $\pp$ in when the minimum cycle distance is achieved.

	Conversely, suppose that $c$ appears after $b$ in each row of $\pp$. We may then express $\pp$ as
		$$ \pp = \mtrx{\LL{a}{z} ~ \LL{\ol{w}_1}{\ol{w}_4}~\LL{b}{b}~ \LL{\ol{w}_2}{\ol{w}_5}~\LL{c}{c}~\LL{\ol{w}_3}{\ol{w}_6}~\LL{z}{a}}.$$
	By following path $\gamma = 0^{k_1}1^{k_2}0^{k_3}1^{k_4}$ where $k_1 = 2 + \#\ol{w}_5+\#\ol{w}_6$, $k_2 = 1+\#\ol{w}_3$, $k_3 = 1 +\#\ol{w}_4$ and $k_4 = 1+\#\ol{w}_2$, we arrive at the standard pair
		$$ \mtrx{\LL{a}{z} ~ \LL{\ol{w}_2}{\ol{w}_5}~\LL{c}{c}~ \LL{\ol{w}_1}{\ol{w}_3}~\LL{b}{b}~\LL{\ol{w}_3}{\ol{w}_6}~\LL{z}{a}}.$$
	There is not a unique shortest path from $\pp$ to this pair. Indeed, we may use $\gamma' = 1^{k'_1}0^{k'_2}1^{k'_3}0^{k'_4}$ instead for $k'_1= 2+\#\ol{w}_2+\#\ol{w}_3$, $k_2' = 1+\#\ol{w}_6$, $k'_3 = 1 + \#\ol{w}_1$ and $k'_4 = 1+\#\ol{w}_5$. We may use either path in the following definition.

  \begin{defn}\label{DefSwitch}
      Let $\pp=(p_0,p_1) \in \std{\AAA}$ and $b,c\in\AAA$ be letters that satisfy the following,
	  $$ 1 < p_\eps(b) < p_\eps(c) < \#\AAA$$
      for $\eps\in\{0,1\}$. The \term{$\{b,c\}$-(inner) switch} (or $\{c,b\}$-(inner) switch) on $\pp$ is a sequence of Rauzy Moves that connects
	  $$ p = \mtrx{\LL{a}{z} ~ \LL{\ol{w}_1}{\ol{w}_4}~\LL{b}{b}~ \LL{\ol{w}_2}{\ol{w}_5}~\LL{c}{c}~\LL{\ol{w}_3}{\ol{w}_6}~\LL{z}{a}} \mbox{ and }
	      \mtrx{\LL{a}{z} ~ \LL{\ol{w}_2}{\ol{w}_5}~\LL{c}{c}~ \LL{\ol{w}_1}{\ol{w}_3}~\LL{b}{b}~\LL{\ol{w}_3}{\ol{w}_6}~\LL{z}{a}}$$
      as explained in the text preceding this definition.
  \end{defn}

	For our work in Extended Rauzy Classes, we instead see that we may move from distinct standard pairs in two cycles. In this case, we choose any $b$ that does not begin or end any row of our standard pair $\pp$. Express $\pp$ as
	$$ \pp = \mtrx{\LL{a}{z}~\LL{\ol{w}_1}{\ol{w}_3}~\LL{b}{b}~\LL{\ol{w}_2}{\ol{w}_4}~\LL{z}{a}}.$$
	By following either path $\gamma = 0^{1+\#\ol{w}_4}\tilde{1}^{1+\#\ol{w}_1}$ or $\gamma'=\tilde{1}^{1+\#\ol{w}_1}0^{1+\#\ol{w}_4}$, we arrive at standard pair
		$$ \mtrx{\LL{a}{b}~\LL{\ol{w}_2}{\ol{w}_4}~\LL{z}{z}~\LL{\ol{w}_1}{\ol{w}_3}~\LL{b}{a}}.$$
	If we instead follow either $\gamma'' = 1^{1+\#\ol{w}_2}\tilde{0}^{1+\#\ol{w}_3}$ or $\gamma''' = \tilde{0}^{1+\#\ol{w}_3}1^{1+\#\ol{w}_2}$, we arrive at pair
		$$ \mtrx{\LL{b}{z}~\LL{\ol{w}_2}{\ol{w}_4}~\LL{a}{a}~\LL{\ol{w}_1}{\ol{w}_3}~\LL{z}{b}}.$$
	For this work the choice of path does not matter.

  \begin{defn}
      Let $\pp = (p_0,p_1) \in \std{\AAA}$ and $b\in\AAA$. Then the \term{$\{a,b\}$-(outer) switch} is the sequence of Rauzy Moves that connects
	  $$ \pp = \mtrx{\LL{a}{z}~\LL{\ol{w}_1}{\ol{w}_3}~\LL{b}{b}~\LL{\ol{w}_2}{\ol{w}_4}~\LL{z}{a}} \mbox{ and }
		\mtrx{\LL{b}{z}~\LL{\ol{w}_2}{\ol{w}_4}~\LL{a}{a}~\LL{\ol{w}_1}{\ol{w}_3}~\LL{z}{b}},$$
      the \term{$\{b,z\}$-(outer) switch} on $\pp$ is the sequence of Rauzy Moves that connects $\pp$ with
	  $$ \mtrx{\LL{a}{b}~\LL{\ol{w}_2}{\ol{w}_4}~\LL{z}{z}~\LL{\ol{w}_1}{\ol{w}_3}~\LL{b}{a}}.$$
      Both moves are detailed in the paragraphs prior to this definition.
  \end{defn}

  \begin{remark}
      We will usually refer to both types of switch moves simply as switch moves when the letters are provided. This will not be ambiguous, as a switch is an outer switch if and only if one of the letters is either the first or last on the top/bottom row. An inner switch will have two letters that appear in the interior of both rows.
  \end{remark}

\subsection{Piece-wise Order Reversing Pairs, Blocks and Chains}

  We call a pair $\pp = (p_0,p_1)\in\perm{\AAA}$ \term{order reversing} if for all $b\in\AAA$,
    \begin{equation}\label{EqOrderReversing}
	p_0(b)+p_1(b) = \#\AAA+1.
    \end{equation}
  Note that all order reversing pairs are standard and therefore irreducible. Not every Rauzy Class contains an order reversing pair. We will be spending the majority of this paper considering such classes, so we will provide a relaxed notion of order reversing. Before we do so, we will adopt a convention of notation.

  When a collection of letters appear (as a set) in same positions in both rows of standard pair $\pp$, we may denote this by a bold letter, such as $\mathbf{B}$. When the letters in $\mathbf{B}$ are to be expressed, we will use square brackets rather than the curved ends for a pair. For example,
    $$ \pp = \mtrx{\LL{a}{z} ~ \mathbf{B} &\LL{d}{e}~\LL{e}{d}&\LL{z}{a}}\mbox{ for } \mathbf{B} = \bmtrx{b&c\\ c&b}$$
    describes the pair
    $$ \pp = \mtrx{\LL{a}{z}~\LL{b}{c}~\LL{c}{b}& \LL{d}{e}~\LL{e}{d}& \LL{z}{a}}.$$

  \begin{defn}\label{DefPWOR}
      A pair $\pp=(p_0,p_1)\in\std{\AAA}$ is \term{piece-wise order reversing} if there exists numbers $2=k_0<k_1<\dots<k_\ell=\#\AAA$ such that:
	\begin{itemize}
	 \item For each $1\leq i \leq \ell$, $p_0^{-1}\{k_{i-1},\dots,k_i-1\} = p_1^{-1}\{k_{i-1},\dots,k_i-1\}$. In other words, the set of letters in positions $k_{i-1}$ to $k_i-1$ for $p_0$ agrees with the same set in $p_1$.
	 \item For each $b$ such that $k_{i-1}\leq p_0(b) < k_i$ for some $1\leq i \leq \ell$,
	      $$ p_0(b)+p_1(b) = k_{i-1}+k_i-1.$$
	\end{itemize}
      We call each set of letters $p_0^{-1}\{k_{i-1},\dots,k_i-1\}$ an \term{order reversing block}. If the number of letters in such a block is $k$, we will refer to this as a \emph{$k$-block}.
  \end{defn}

  So a pair $\pp = (p_0,p_1)$ is piece-wise order reversing if it is standard and each set of letters appear in groups whose order within the group is reversed from $p_0$ to $p_1$. In proofs that deal with piece-wise order reversing pairs, we will adopt a particular convention useful to simplify arguments.

  \begin{defn}
      Let $\pp\in\std{\AAA}$ be piece-wise order reversing. A \term{chain} in $\pp$ will be a (possibly empty) collection of consecutive order-reversing blocks. If $\mathbf{S}_1,\dots,\mathbf{S}_{k-1}$ denote the $1$-blocks (order reversing blocks of length one), we call the \term{chain decomposition} of $\pp$ the ordered list of chains $\mathbf{C}_1,\dots,\mathbf{C}_k$ such that $\mathbf{C}_i$ is the chain preceding $\mathbf{S}_i$ and following $\mathbf{S}_{i-1}$, using the convention that $\mathbf{C}_1$ is simply the chain preceding $\mathbf{S}_1$ and $\mathbf{C}_k$ is the chain following $\mathbf{S}_{k-1}$.
  \end{defn}

  \begin{remark}\label{RemIgnoringChains}
      Once we provide a chain decomposition for $\pp$, the following lemmas allow us to perform switch moves on only a few chains while safely ignoring all others. In particular, Lemma \ref{LemReorderingChains} allows us to move any choice of chains (excluding the rightmost chain) to the left of all others. Lemma \ref{LemIgnoringChains} then says that, by possibly performing an extra switch move afterward, all sequences of switch moves that alter the rightmost chains may ignore the chains to their left. We will often abuse notation and refer to a pair by only the relevant chains, with the understanding that we always can replace the omitted chains without alteration.
  \end{remark}

  \begin{lem}\label{LemReorderingChains}
      Suppose $\pp$ is a piece-wise order reversing pair decomposed into chains
			$$ \textbf{C}_1,\dots , \textbf{C}_k.$$
      Then for any reordering $\{i_1,\dots,i_{k-1}\} = \{1,\dots, k-1\}$, the piece-wise order reversing pair $\pp'$ with chain decomposition
	$$ \textbf{C}_{i_1},\dots, \textbf{C}_{i_{k-1}}, \textbf{C}_k$$
      is connected to $\pp$ by inner switch moves.
  \end{lem}

  \begin{proof}
    Let $s_j$ be the letter forming the $1$-block between $\textbf{C}_j$ and $\textbf{C}_{j+1}$, $1\leq j < k$. By performing a $\{s_{i_{k-1}},s_{k-1}\}$-switch, we arrive at a pair with chain decomposition
	$$ \textbf{C}_{i_{k-1}+1},\dots,\textbf{C}_{k-1},\textbf{C}_1,\dots, \textbf{C}_{i_{k-1}}, \textbf{C}_k.$$
    Note that the index of a letter $s_j$ still matches the index of the preceding chain $\textbf{C}_j$. We continue iteratively. Suppose $\textbf{C}_{i_\ell},\dots,\textbf{C}_{i_{k-1}},\textbf{C}_k$ are in correct position in our pair. If $\textbf{C}_j$ is the chain preceding $\textbf{C}_{i_\ell}$, $j\neq i_{\ell-1}$, we may perform an $\{s_j,s_{i_{\ell-1}}\}$-switch to move $\textbf{C}_{i_{\ell-1}}$ into the correct position.

    This process terminates when all chains are in the appropriate position.
  \end{proof}

  \begin{lem}\label{LemIgnoringChains}
    Suppose $\pp\in\std{\AAA}$ is piece-wise order reversing with chain decomposition
	$$ \textbf{C}_1,\dots , \textbf{C}_k.$$
    Suppose also that $\pp'=\pp\RED{\AAA'}$ for $\AAA'$ such that $\pp'$ is still piece-wise order reversing with decomposition
	$$ \textbf{C}_j,\dots , \textbf{C}_k$$
    for some $j>1$. If $\qq'$ is connected to $\pp'$ by a sequence of inner switch moves and has chain decomposition
	$$ \mathbf{D}_1,\dots,\mathbf{D}_m,$$
    then there is a piece-wise order reversing pair $\qq$ connected to $\pp$ by switch moves with chain decomposition
	$$ \mathbf{C}_1,\dots,\mathbf{C}_{j-1},\mathbf{D}_1,\dots,\mathbf{D}_m.$$
  \end{lem}

  \begin{proof}
      We will consider $\pp$, $\pp'$ and $\qq'$ as in the statement of the theorem. Let $s$ denote the letter that forms the $1$-block to the immediate left of chain $\mathbf{C}_j$ in $\pp$, and let $\BBB = \AAA'\sqcup\{s\}$. By Lemma \ref{LemIrreducibleExtensions}, $\pp^{*}=\pp\RED{\BBB}\in\irr{\BBB}$. It also follows that $p^*_0(s) = p_1^*(s) = 2$ where $\pp^* = (p^*_0,p^*_1)$. We see then that $\Sof{\pp^*}(s) = s$, from Definition \ref{DefSofPiAndMofPi}. By Lemma \ref{LemReductionPaths}, the switch moves connecting $\pp'$ to $\qq'$ extend to switch moves connecting $\pp^*$ to $\qq^*$ where $\qq^*= \qq\RED{\AAA'}$. If $q^*_0(s) = q^*_1(s) = 2$ where $\qq^* = (q^*_0,q^*_1)$, then these switch moves also extend to switch moves from $\pp$ to the appropriate $\qq$.

      If $s$ is not in the correct position, $\Sof{\qq^*}(s) = s$ by Proposition \ref{PropSofPiIsInvariant}. This means that there must exist a letter $c\in\BBB$ such that $q_\eps^*(c) +1 = q_\eps^*(s)$ for $\eps\in\{0,1\}$. By performing a $\{c,s\}$-switch on $\qq^*$, we arrive at $\rr^* = (r_0^*,r_1^*)$ such that $r_0^*(s) = r_1^*(s) = 2$. The sequence of switch moves from $\pp^*$ to $\qq^*$ to $\rr^*$ will extend to a sequence of switch moves that connect $\pp$ to the desired $\qq$.
  \end{proof}

\section{Finding Piece-wise Order Reversing Pairs in Each Class}\label{SecPWOR}

  In order to classify all Rauzy Classes, we will need to find a convenient pair form in each Rauzy Class that will allow for further manipulation. We have already noted that order reversing pairs do not exist in every (in fact most!) Rauzy Classes. However, we will show that every Rauzy Class does contain a piece-wise order reversing pair, as given in Definition \ref{DefPWOR}. Given the notion of inverse for pairs that appears before Equation \eqref{EqInverseMoves}, every piece-wise order reversing pair is its own inverse. In this way, Theorem \ref{ThmPWOrderReversing} below now provides a purely combinatorial proof of \cite[Theorem 2.1]{cFick2014}: Every Rauzy Class contains a self-inverse pair.

  The proof of this theorem follows very much in spirit to the simple proof provided for Proposition \ref{PropStandardInClass}.

  \begin{thm}\label{ThmPWOrderReversing}
    Every Rauzy Class contains a piece-wise order reversing pair.
  \end{thm}

  \begin{proof}
      We shall provide a procedure to connect any $\pp\in\irr{\AAA}$ to a piece-wise order reversing pair by switch moves. By Proposition \ref{PropStandardInClass}, we may first assume that $\pp$ is standard.

      Suppose that $\pp=(p_0,p_1)$ is not piece-wise order reversing. Let $n<\#\AAA$ be the unique value such that all letters of the form $p_\eps^{-1}(\ell)$, $\eps\in\{0,1\}$ and $n < \ell < \#\AAA$, belong to order reversing blocks while $p_0^{-1}(n)$ and $p_1^{-1}(n)$ do not. Note that $n>1$ if $\pp$ is not piece-wise order reversing. Let $b_\eps = p_\eps^{-1}(n)$ denote the distinct letters that appear in position $n$ in row $\eps$ for $\eps\in\{0,1\}$. Let $n_0 = p_0(b_1)$ and $n_1 = p_1(b_0)$. By definition of $n$, $n_\eps< n-1$ for some $\eps\in\{0,1\}$. We will proceed by cases.

      \textbf{Case 1:} If $n_0 \neq n_1$, assume without loss of generality that $n_0 < n_1$. Then there must exist a letter $c\in\AAA$ such that $n_0 < p_0(c) < n$ and $p_1(c) < n_1$. By performing a $\{b_0,c\}$-switch to arrive at pair $\pp'$ (see Figure \ref{FigProofOfPWOR}). Either $\pp'$ is now piece-wise order reversing, or the following hold: $n' < n$ or $n' = n$, $b_0' = c$, $b_1'=b_1$ and $n_0'>n_0$.

      \textbf{Case 2:} Suppose $n_0 = n_1$, but a letter $c$ still exists such that $n_0 < p_0(c) < n$ and $p_1(c) < n_1$. We may perform the same switch move as in Case 1 to arrive at $\pp'$ with the same properties, specifically $n_0' > n_0$. A similar argument follows if $n_1 < p_1(c) < n$ and $p_0(c) < n_0$ instead.

      \textbf{Case 3:} Suppose $n_0=n_1$ and no such letter (from Case 2) exists. Because $b_0$ and $b_1$ do not belong to an order-reversing block, there must exist letters $c$ and $d$ such that $n_\eps < p_\eps(c) < p_\eps(d) < n$ for $\eps\in\{0,1\}$. By performing a $\{c,d\}$-switch, we arrive at pair $\pp''$. Either $\pp''$ is piece-wise order reversing, $n'' < n$ or $n''= n$, $n_0''>n_0$ and $n_1''>n_1$.

      In any case, we perform a switch move to replace $\pp$ with new pair $\pp'$ or $\pp''$. This new pair will either be piece-wise order reversing, have a decreased value of $n$ or increased value of $n_0$ and/or $n_1$. Because each of these steps change our pair on finite values, this process will complete at a piece-wise order reversing pair.
    
      \begin{figure}[t]
		$$ \begin{array}{|lr|}
		      \hline
		      \mathrm{Cases~ 1\&2:} & \xymatrix{\pp = \mtrx{\LL{a}{z}\LL{\ol{u}_1}{\ol{u}_2}\LL{b_1}{c}\LL{\ol{v}_1}{\ol{v}_2}\LL{c}{b_0}\LL{\ol{w}_1}{\ol{w}_2}\LL{b_0}{b_1}\LL{\ol{x}_1}{\ol{x}_2}\LL{z}{a}} \ar@{.}[r]^{\{b_0,c\}} &
			\mtrx{\LL{a}{z}\LL{\ol{w}_1}{\ol{v}_2}\LL{b_0}{b_0}\LL{\ol{u}_1}{\ol{u}_2}\LL{b_1}{c}\LL{\ol{v}_1}{\ol{w}_2}\LL{c}{b_1}\LL{\ol{x}_1}{\ol{x}_2}\LL{z}{a}} = \pp'}  \\ & \\
		      \mathrm{Case~ 3:}	    & \xymatrix{\pp = \mtrx{\LL{a}{z}\LL{\ol{u}_1}{\ol{u}_2}\LL{b_1}{b_0}\LL{\ol{v}_1}{\ol{v}_2}\LL{c}{c}\LL{\ol{w}_1}{\ol{w}_2}\LL{d}{d}\LL{\ol{x}_1}{\ol{x}_2}\LL{b_0}{b_1}\LL{\ol{y}_1}{\ol{y}_2}\LL{z}{a}} \ar@{.}[r]^{\{c,d\}} &
			\mtrx{\LL{a}{z}\LL{\ol{w}_1}{\ol{w}_2}\LL{d}{d}\LL{\ol{u}_1}{\ol{u}_2}\LL{b_1}{b_0}\LL{\ol{v}_1}{\ol{v}_2}\LL{c}{c}\LL{\ol{x}_1}{\ol{x}_2}\LL{b_0}{b_1}\LL{\ol{y}_1}{\ol{y}_2}\LL{z}{a}} = \pp''} \\ \hline
		   \end{array}  $$
	  \caption{Proving Theorem \ref{ThmPWOrderReversing}: Performing switch moves in each case.}\label{FigProofOfPWOR}
      \end{figure}
  \end{proof}

\section{Proofs of the Kontsevich-Zorich-Boissy Classification}\label{SecKZB}

  The classification theorems provided in \cite{cKonZor2003} and \cite{cBoi2012} use two concepts. We have provided the combinatorial version of one concept by our definitions of $\Mof{\pp}$ and $\Pof{\pp}$ in Section \ref{SecBackground}. The remaining concept will be covered by the following definition.

  \begin{defn}\label{DefTypes}
     A piece-wise order reversing pair is:
	  \begin{itemize}
                \item[] \term{Type} $<1>$ if there exists at most one non-empty chain which, if it exists, is composed of one block.
		\item[] \term{Type} $<2>$ if all non-empty chains are of the composed of $2$-blocks.
		\item[] \term{Type} $<3>$ if each non-empty chain is
				\begin{itemize}
					\item[] (i) composed of $m$ consecutive $2$-blocks followed by a $3$-block and ending with $\ell$ consecutive $2$-blocks, $m,\ell\geq 0$.
					\item[] (ii) composed of $m$ consecutive $2$-blocks, $m>0$.
				\end{itemize}
			and at least one word must be of the first form.
		\item[] \term{Type} $<4>$ if exactly one chain is composed of a $4$-block followed by $m$ consecutive $2$-blocks ($m\geq 0$) and all other non-empty chains composed of $2$-blocks.
		\item[] \term{Type} $<5>$ if exactly one chain is composed of a $5$-block and all other non-empty chains are composed of exactly one $2$-block each.
          \end{itemize}
     Many piece-wise order reversing pairs do not have a Type. To ensure disjointness, Types $<2>$--$<5>$ all require that more than one block of size two or greater exists in the pair.
  \end{defn}

  Theorems \ref{ThmTypes} and \ref{ThmExclassUniqueType} will allow us to assign a unique Type to each (Extended) Rauzy Class. All results have purely combinatorial proofs.

  We then achieve an identical result to \cite{cKonZor2003} and \cite{cBoi2012}. By Theorem \ref{ThmExclassTypeSufficient} Type and $\Pof{\pp}$ fully classify all non-labeled Extended Rauzy Classes, which shows \cite[Theorems 1 \&\ 2]{cKonZor2003}. By Theorem \ref{ThmClassTypeSufficient}, $\Pof{\pp}$, $\Mof{\pp}$ and Type fully classify all non-labeled Rauzy Classes, proving the Main Theorem (for Translation Surfaces) in \cite{cBoi2012}.

\subsection{Every Class Contains a Pair of Type 1,2,3,4 or 5}

    Given a Rauzy Class, Theorem \ref{ThmPWOrderReversing} will give a piece-wise order reversing pair $\pp$ contained in the class. We will work in steps, represented by lemmas. If $\pp$ is Type $<1>$, then we are already satisfied. If $\pp$ is not Type $<1>$, we will apply Lemma \ref{LemSize5} to get a new pair with all blocks size five or less. We will then apply Lemma \ref{LemUniqueSize5} to get a piece-wise order reversing pair with at most block of size $4$ or $5$ with all other blocks size at most $3$. We then finish the proof of Theorem \ref{ThmTypes} with the final refinements to our pair, making it one of Types $<2>$--$<5>$.

    \begin{lem}\label{LemSize5}
	Every Rauzy Class contains a pair that is either piece-wise order reversing with all order reversing blocks of length at most five or of Type $<1>$.
    \end{lem}
    
    \begin{proof}
	By Theorem \ref{ThmPWOrderReversing}, we may find $\pp$ in a Rauzy Class that is piece-wise order reversing. Let us assume that $\pp$ is not Type $<1>$, otherwise the proof is complete. We will then perform switch moves to find a pair that has blocks of sufficient size. In each of the following cases, we will replace $\pp$ with a new pair $\pp'$ that will have replaced a large block with smaller ones. This proof will have to be used iteratively, so we will always assume that our block of size at least $6$ is the rightmost of such blocks.

	\textbf{Case 1:} Suppose that $\pp$ contains an order-reversing block of size $m\geq 8$. Because $\pp$ is not Type $<1>$, it must contain another order reversing block of size $k \geq 2$.
      
	\textbf{Case 1A:} \textit{Suppose this $k$-block appears to the left of our $m$-block.} We may then express $\pp$ as
	    $$ \pp = \mtrx{\LL{a}{z}~ \LL{\ol{u}_1}{\ol{u}_2} ~ \LL{b_1~b_2~b_3~b_4~ \ol{v}_1~b_5~b_6~b_7~b_8}{b_8~b_7~b_6~b_5~ \ol{v}_2~b_4~b_3~b_2~b_1}~ \LL{\ol{w}_1}{\ol{w}_2}~\LL{c_1\cdots c_k}{c_k\cdots c_1}~\LL{\ol{x}_1}{\ol{x}_2}~\LL{z}{a}}$$
	where $b_1,\dots,b_8,\ol{v}_1,\ol{v}_2$ form the size $m$ block and $c_1,c_2,\dots,c_k$ form our block of size $k\geq 2$. In other words, $\ol{v}_1$ and $\ol{v}_2$ are (possibly empty) words that are reverses of each other. By performing the following switches in order:
	    $$ \{b_3,c_k\},\{b_7,c_1\},\{b_2,b_5\},\{b_2,b_8\},\{b_4,b_6\},\{b_1,b_4\},\{b_7,c_1\},\{b_5,b_8\},\{b_3,c_k\},$$
	we arrive at pair $\pp'$ of the form
	    $$ \pp' = \mtrx{\LL{a}{z}~ \LL{\ol{u}_1}{\ol{u}_2} ~ \LL{b_1~b_8~b_3~b_6}{b_8~b_1~b_6~b_3}~ \LL{\ol{v}_1}{\ol{v}_2} ~\LL{b_5~b_4~b_7~b_2}{b_4~b_5~b_2~b_7}~ \LL{\ol{w}_1}{\ol{w}_2}~\LL{c_1\cdots c_k}{c_k\cdots c_1}~\LL{\ol{x}_1}{\ol{x}_2}~\LL{z}{a}}.$$
	We note that $\pp'$ is piece-wise order reversing and has replaced the $m$-block in $\pp$ with four $2$-blocks and an $(m-8)$-block.

	\textbf{Case 1B:} Let us now assume that our $k$-block is to the left of our $m$-block. We then express $\pp$ as
	    $$ \pp = \mtrx{\LL{a}{z}~ \LL{\ol{u}_1}{\ol{u}_2} ~ \LL{c_1\cdots c_k}{c_k\cdots c_1}~ \LL{\ol{v}_1}{\ol{v}_2}~\LL{b_1~b_2~b_3~b_4~ \ol{w}_1~b_5~b_6~b_7~b_8}{b_8~b_7~b_6~b_5~ \ol{w}_2~b_4~b_3~b_2~b_1}~\LL{\ol{x}_1}{\ol{x}_2}~\LL{z}{a}},$$
	with the $b_i$'s and words $\ol{w}_1,\ol{w}_2$ forming the $m$-block and the $c_i$'s forming the $k$-block.
	By performing the following switches in order:
	    $$ \{c_k,b_2\},\{c_1,b_5\},\{b_1,b_4\},\{b_1,b_6\},\{c_1,b_3\},\{c_k,b_5\},\{b_2,b_8\},\{b_4,b_7\},\{b_1,b_4\},$$
	we arrive at pair $\pp'$ of the form
	    $$\pp' = \mtrx{\LL{a}{z}~ \LL{\ol{u}_1}{\ol{u}_2} ~ \LL{c_1\cdots c_k}{c_k\cdots c_1}~ \LL{\ol{v}_1}{\ol{v}_2}~\LL{b_1~b_8~b_3~b_6}{b_8~b_1~b_6~b_3}~\LL{\ol{w}_1}{\ol{w}_2}~\LL{b_5~b_4~b_7~b_2}{b_4~b_5~b_2~b_7}~\LL{\ol{x}_1}{\ol{x}_2}~\LL{z}{a}},$$
	which has replaced the $m$-block of $\pp$ with four $2$-blocks and an $(m-8)$-block. In either case, we may replace $\pp$ with $\pp'$ and repeat this argument until we have a permutation with all blocks size less than $8$.

	\textbf{Case 2:} Suppose $\pp$ is piece-wise order-reversing with a $7$-block and another block of size $k\geq 2$. As with the previous case, we assume that $7$-block is the rightmost block of size $7$. We consider two cases: either our $7$-block is the leftmost block of size at least $2$, or there exists a block of size $k\geq2$ to its left.

	\textbf{Case 2A:} In our first case, we may express $\pp$ as
	    $$ \pp = \mtrx{\LL{a}{z}~\LL{\ol{u}~s}{\ol{u}~s}~\LL{b_1~b_2~b_3~b_4~b_5~b_6~b_7}{b_7~b_6~b_5~b_4~b_3~b_2~b_1}~ \LL{\ol{v}_1}{\ol{v}_2}~\LL{c_1\cdots c_k}{c_k \cdots c_1}~\LL{\ol{w}_1}{\ol{w}_2} ~\LL{z}{a}},$$
	where $b_1,\dots,b_7$ form our $7$-block, $c_1,\dots,c_k$ form our $k$-block, and $\ol{u}$,$s$ represent all $1$-blocks to the left of the $7$-block. We may without difficulty also consider the case when $\ol{u}$ and $s$ do not exist, meaning the $7$-block would be the leftmost block (of any size) of $\pp$. By performing the following switch moves:
	    $$ \{b_5,c_1\},\{b_1,c_k\},\{b_4,b_6\},\{b_3,c_1\},\{b_5,c_1\},\{b_1,c_k\},\{b_2,b_7\},\{s,b_3\},$$
	excluding the final move if $s$ does not exist, we arrive at pair
	    $$ \pp' = \mtrx{\LL{a}{z}~\LL{\ol{u}~s}{\ol{u}~s}~\LL{b_3~b_1}{b_1~b_3}~\LL{b_5~b_4~b_7}{b_7~b_4~b_5}~\LL{b_6~b_2}{b_2~b_6}~ \LL{\ol{v}_1}{\ol{v}_2}~\LL{c_1\cdots c_k}{c_k \cdots c_1}~\LL{\ol{w}_1}{\ol{w}_2}~ \LL{z}{a}},$$
	replacing the $7$-block of $\pp$ with two $2$-blocks and one $3$-block in $\pp'$. The second case shall be more involved.

	\textbf{Case 2B:} Suppose the $7$-block instead has a block of size at least 2 to its left. We choose the leftmost such block and express $\pp$ as
	    $$ \pp = \mtrx{\LL{a}{z} ~\LL{\ol{u}~s}{\ol{u}~s}~ \LL{c_1 \cdots c_k}{c_k\cdots c_1}~\LL{\ol{v}_1}{\ol{v}_2}~\LL{b_1~b_2~b_3~b_4~b_5~b_6~b_7}{b_7~b_6~b_5~b_4~b_3~b_2~b_1}~\LL{\ol{w}_1}{\ol{w}_2}~ \LL{z}{a}},$$
	where the letters $c_1,\dots,c_k$ form our $k$-block for $k\geq 2$ and $b_1,\dots,b_7$ form our $7$-block. The word $\ol{u}$ and letter $s$ are present if there are $1$-blocks to the left of our $k$-block. We may again assume that $\ol{u}$ is empty and $s$ does not exist if needed. By performing the following moves in order:
	    $$ \{c_1,b_5\},\{c_k,b_1\},\{b_2,b_7\},\{b_3,b_5\},\{b_6,c_k\},\{b_3,b_4\},\{s,b_4\} ,$$
	excluding the final move if $s$ does not exist, we arrive at pair
	    $$ \pp' = \mtrx{\LL{a}{z} ~\LL{\ol{u}~s}{\ol{u}~s}~ \LL{b_5 ~ b_3}{b_3~b_5}~\LL{b_7~\ol{v}_1~b_1}{b_1~\ol{v}_2~b_7}~\LL{c_2\cdots c_{k-1}}{c_{k-1}\cdots c_2}~\LL{c_k~b_4~c_1}{c_1~b_4~c_k}~\LL{b_6~b_2}{b_2~b_6}~\LL{\ol{w}_1}{\ol{w}_2}~ \LL{z}{a}}.$$
	Unless $\ol{v}_1$ and $\ol{v}_2$ represent exactly one order reversing block, this pair is no longer piece-wise order-reversing. However, we may consider the pair $\pp'$ restricted to the alphabet $\AAA' := \AAA\setminus\big(\{b_2, b_4,b_6,c_1,\dots,c_k,\}\cup\{\alpha \in \AAA: \alpha \in \ol{w}_1\}\big)$ and let $\qq = \pp'\RED{\AAA'}$. The removed letters form blocks of size at most $6$, as we always pick the rightmost as the beginning of our proof and have finished Case 1. So we may find $\qq'\in\ClassLab{\qq}$ that is piece-wise order reversing by Theorem \ref{ThmPWOrderReversing}. We may then apply the proof of this Theorem to $\qq'$ to find a pair $\qq''$ which is piece-wise order reversing and either is Type $<1>$ or has all blocks size at most $6$. This extends to $\pp''$ in the class of $\pp$ and $\pp'$ that is also piece-wise order-reversing. This pair either will have all blocks the correct size, or has exactly \textbf{one} block size at least $6$. This block will be the leftmost block of size greater than $1$, so our proof will work directly for either Case 1 or 2. We note that our restriction always removes at least 5 letters from our alphabet, so we will only have to reduce up to a finite number of times.

	\textbf{Case 3:} Assume now that $\pp$ is piece-wise order-reversing with a $6$-block and another block of size $k\geq 2$.

	\textbf{Case 3A:} If the $k$-block is to the right of the $6$ block, we may express $\pp$ as
	    $$ \pp = \mtrx{\LL{a}{z}~\LL{\ol{u}_1}{\ol{u}_2}~\LL{b_1~b_2~b_3~b_4~b_5~b_6}{b_6~b_5~b_4~b_3~b_2~b_1}~\LL{\ol{v}_1}{\ol{v}_2}~\LL{c_1\cdots c_k}{c_k\cdots c_1}~\LL{\ol{w}_1}{\ol{w}_2}~\LL{z}{a}}.$$
	By performing the ordered switch moves:
	    $$ \{b_1,c_k\},\{b_4,c_1\},\{b_1,b_3\},\{b_1,b_5\},\{b_2,b_6\},\{b_4,c_1\},\{b_3,c_k\},$$
	we arrive at pair $\pp'$ of the form
	    $$ \pp' = \mtrx{\LL{a}{z}~\LL{\ol{u}_1}{\ol{u}_2}~\LL{b_1~b_6}{b_6~b_1}~\LL{b_3~b_2~b_5~b_4}{b_4~b_5~b_2~b_3}~\LL{\ol{v}_1}{\ol{v}_2}~\LL{c_1\cdots c_k}{c_k\cdots c_1}~\LL{\ol{w}_1}{\ol{w}_2}~\LL{z}{a}},$$
	which has replaced the $6$-block of $\pp$ with a $2$-block and a $4$-block.

	\textbf{Case 3B:} If instead the $k$-block of $\pp$ is to the left of the $6$-block, then we write $\pp$ as
	    $$ \pp = \mtrx{\LL{a}{z}~\LL{\ol{u}_1}{\ol{u}_2}~\LL{c_1\cdots c_k}{c_k\cdots c_1}~\LL{\ol{v}_1}{\ol{v}_2}~\LL{b_1~b_2~b_3~b_4~b_5~b_6}{b_6~b_5~b_4~b_3~b_2~b_1}~\LL{\ol{w}_1}{\ol{w}_2}~\LL{z}{a}}.$$
	By performing the following switch moves:
	    $$\{b_3,c_1\},\{b_1,c_k\},\{b_1,b_5\}, \{b_2,b_4\},\{b_4,b_6\},\{b_5,c_1\},\{b_3,c_k\},$$
	we arrive a $\pp'$ of the form
	    $$ \mtrx{\LL{a}{z}~\LL{\ol{u}_1}{\ol{u}_2}~\LL{c_1\cdots c_k}{c_k\cdots c_1}~\LL{\ol{v}_1}{\ol{v}_2}~\LL{b_1~b_6}{b_6~b_1}~\LL{b_3~b_2~b_5~b_4}{b_4~b_5~b_2~b_3}~\LL{\ol{w}_1}{\ol{w}_2}~\LL{z}{a}},$$
	which has again replaced the $6$-block of $\pp$ with a $2$-block and a $4$-block.
    
	We have therefore shown steps to iteratively replace large blocks in $\pp$ with smaller blocks. This process is finite and will terminate when we achieve a pair $\pp$ as stated in the theorem.
    \end{proof}
    
    \begin{lem}\label{LemUniqueSize5}
		Every Rauzy Class contains a pair that is either Type $<1>$ or is piece-wise order reversing with at most one order-reversing block of length $4$ or $5$ and all other blocks of length at most $3$.
    \end{lem}
    
    \begin{proof}
		Given a Rauzy Class, we may find a piece-wise order reversing pair $\pp$ satisfying the conditions of Lemma \ref{LemSize5}. We may assume that $\pp$ is not Type $<1>$, so it is piece-wise order-reversing with all blocks of size at most $5$. Much like in the previous proof, we will reduce the number of blocks of size at least $4$ by replacing $\pp$ with $\pp'$. We will perform this operation iteratively until we arrive at our desired permutation.
		
		Suppose $\pp$ has a $k$-block and an $m$-block, where $k,m\in\{4,5\}$. We then express $\pp$ as
			$$ \pp = \mtrx{\LL{a}{z}~\LL{\ol{u}_1}{\ol{u}_2}~\LL{b_1~b_2~b'~b_3~b_4}{b_4~b_3~b'~b_2~b_1}~\LL{\ol{v}_1}{\ol{v}_2}~\LL{c_1~c_2~c'~c_3~c_4}{c_4~c_3~c'~c_2~c_1}~\LL{\ol{w}_1}{\ol{w}_2}~\LL{z}{a}}$$
		where $b'$ exists if $k=5$ and $c'$ exists if $m=5$. By performing Move \ref{Move-43218765-to-21436587}, we arrive at pair
			$$ \pp' = \mtrx{\LL{a}{z}~\LL{\ol{u}_1}{\ol{u}_2}~\LL{b_1~b_4}{b_4~b_1}~\LL{b'}{b'}~\LL{b_3~b_2}{b_2~b_3}~\LL{\ol{v}_1}{\ol{v}_2}~\LL{c_1~c_4}{c_4~c_1}~\LL{c'}{c'}~\LL{c_3~c_2}{c_2~c_3}~\LL{\ol{w}_1}{\ol{w}_2}~\LL{z}{a}}$$
		which has replaced the two larger blocks with blocks of size at most $2$.
    \end{proof}

    \begin{thm}\label{ThmTypes}
      Every Rauzy Class contains a pair of one of the five Types in Definition \ref{DefTypes}.
    \end{thm}
    
    \begin{proof}
		By Lemma \ref{LemUniqueSize5}, we may consider $\pp$ in our Rauzy Class that is either already Type $<1>$ or has at most one block of size at most $5$ and all other blocks of size at most $3$. Let $\mathbf{C}_1,\dots,\mathbf{C}_k$ be the chain decomposition of $\pp$. We will perform moves on each chain based on cases to arrive at a new pair $\pp'$. After each move, we replace $\pp$ with $\pp'$. As stated in Remark \ref{RemIgnoringChains}, we will express only the chains that will be altered.
		
		\textbf{Case 1A:} \textit{$\pp$ contains only $1$-blocks and $2$-blocks.} In this case, $\pp$ is Type $<2>$.

		\textbf{Case 1B:} \textit{$\pp$ does not contain $3$-blocks but does contain a $4$-block.} Let $\mathbf{C}_j$ be the chain in $\pp$ that contains the $4$-block, and note that each other non-empty chain is composed of $2$-blocks. If $\mathbf{C}_j$ is not of the form listed for Type $<4>$, there is at least one $2$-block to the left of the $4$-block in $\mathbf{C}_j$. We express $\pp$ as
			$$ \pp = \mtrx{\LL{a}{z}~\LL{\ol{u}_1}{\ol{u}_2}~\LL{b_1~b_2}{b_2~b_1}~\LL{c_1~c_2~c_3~c_4}{c_4~c_3~c_2~c_1}~\LL{\ol{v}_1}{\ol{v}_2}~\LL{z}{a}}.$$
		By performing Move \ref{Move-216543-to-432165}, we arrive at pair $\pp'$ of the form
			$$ \pp' = \mtrx{\LL{a}{z}~\LL{\ol{u}_1}{\ol{u}_2}~\LL{b_1~c_4~c_1~b_2}{b_2~c_1~c_4~b_1}~\LL{c_3~c_2}{c_2~c_3}~\LL{\ol{v}_1}{\ol{v}_2}~\LL{z}{a}}.$$
		We have replaced $\mathbf{C}_j$ with chain $\mathbf{C}'_j$ with the $4$-block moved to the left. By repeating this process as necessary, we will arrive at a pair $\pp$ of Type $<4>$.

		\textbf{Case 1C:} \textit{$\pp$ does not contain $3$-block but does contain a $5$-block.} Let $\mathbf{C}_j$ be the chain containing the $5$-block. If $\mathbf{C}_j$ is composed of only the $5$-block and all other non-empty chains are of composed of exactly one $2$-block each, then $\pp$ is Type $<5>$. If $\mathbf{C}_j$ contains at least one $2$-block, we consider the following two pairs:
			    $$ \qq = \mtrx{\LL{a~\ol{u}_1}{z~\ol{u}_2}~\LL{b_1~c_2~s_1~c_1~b_2}{b_2~c_1~s_1~c_2~b_1}~\LL{s_2}{s_2}~\LL{d_1~d_2}{d_2~d_1}~\LL{\ol{v}_1~z}{\ol{v}_2~a}}$$
		and
			    $$ \rr = \mtrx{\LL{a~\ol{u}_1}{z~\ol{u}_2}~\LL{b_1~b_2}{b_2~b_1}~\LL{s_1}{s_1}~\LL{c_1~d_2~s_2~d_1~c_2}{c_2~d_1~s_2~d_2~c_1}~\LL{\ol{v}_1~z}{\ol{v}_2~a}}$$
		which are related by Move \ref{Move-216543-to-432165} on letters $a,b_1,b_2,c_1,c_2,d_1,d_2,z$. If there is a $2$-block to the right of the $5$-block in $\mathbf{C}_j$, then we may express $\pp$ as the form $\qq$ by eliminating $s_2$. The resulting pair $\pp'$ will be of the form $\rr$ with $s_2$ removed. This has replaced $\mathbf{C}_j$ with two new chains, one with a $4$-block and the other with only $2$-blocks. We may then proceed with $\pp'$ as in Case 1B. If a $2$-block is to the left, we may instead express $\pp$ in the form of $\rr$ by removing $s_1$. By performing the switch moves in reverse order, the resulting pair $\pp'$ may be expressed in the form of $\qq$ with $s_1$ removed. This has replaced $\mathbf{C}_j$ with two chains, again one with a $4$-block and the other with only $2$-blocks.
	  
		The only other possibility is for $\mathbf{C}_j$ to be composed of only the $5$-block and another chain $\mathbf{C}_{j'}$ exists composed of at least two $2$-blocks. We may again consider two pairs
		    $$ \qq = \mtrx{\LL{a~\ol{u}_1}{a~\ol{u}_2}~\LL{b_1~c_2~s_1~c_1~b_2}{b_2~c_1~s_1~c_2~b_1}~\LL{\ol{v}_1}{\ol{v}_2}~\LL{d_1~d_2}{d_2~d_1}~\LL{s_2}{s_2}~\LL{e_1~e_2}{e_2~e_1}~\LL{\ol{w}_1~z}{\ol{w}_2~a}}$$
		and
		    $$ \rr = \mtrx{\LL{a~\ol{u}_1}{a~\ol{u}_2}~\LL{b_1~b_2}{b_2~b_1}~\LL{s_1}{s_1}~\LL{c_1~c_2}{c_2~c_1}~\LL{\ol{v}_1}{\ol{v}_2}~\LL{d_1~e_2~s_2~e_1~d_2}{d_2~e_1~s_2~e_2~d_1}~\LL{\ol{w}_1~z}{\ol{w}_2~a}}$$
		which are connected by Move \ref{Move-43216587-to-21438765} on letters $a,b_1,b_2,c_1,c_2,d_1,d_2,e_1,e_2,z$. If $\mathbf{C}_{j'}$ is to the right of $\mathbf{C}_j$, then we may express $\pp$ as $\qq$ by removing $s_2$. The resulting pair $\pp'$ may be expressed as $\rr$ with $s_2$ removed and has replaced $\mathbf{C}_j$ with two chains both composed of $2$-blocks (if non-empty) and $\mathbf{C}_{j'}$ with a chain composed of a $4$-block and possibly $2$-blocks. This may be handled by Case 1B. If instead $\mathbf{C}_{j'}$ is to the left, we may instead reverse this process by expressing $\pp$ as $\rr$ with $s_1$ removed. The resulting pair may be expressed as $\qq$ with $s_1$ removed and may also be handled by Case 1B.
		
		\textbf{Case 2A:} \textit{$\pp$ contains at least one $3$-block, all blocks are size at most $3$ and each chain contains at most one $3$-block.} This pair is Type $<3>$.	  

		\textbf{Case 2B:} \textit{$\pp$ contains at least one $3$-block and a $k$-block, $k=4$ or $k=5$.} Suppose first that the $k$-block is to the right of a $3$-block. This is in general not difficult to ensure. If $\pp$ does not already satisfy this condition but instead has a chain $\mathbf{C}_j$ that contains a $3$-block and another chain $\mathbf{C}_{j'}$ that contains the $k$-block, then we may switch the ordering of the chains by Lemma \ref{LemReorderingChains} to satisfy this assumption, as long as $\mathbf{C}_j$ is not the rightmost chain. If $\mathbf{C}_j$ is the rightmost chain but has two or more consecutive $2$-blocks to the right of the $3$-block, we may perform a move as in the final paragraph of Case 1C to move the $k$-block into $\mathbf{C}_j$ to the right of the $3$-block, and it will now be a $4$-block (regardless of $k$). So assume we have made any such moves if possible to satisfy our assumption.

		Choose our $3$-block so that it is the rightmost $3$-block in its chain. We may then express $\pp$ as
		      $$ \pp = \mtrx{\LL{a}{z}~\LL{\ol{v}_1~b_1~b_2~b_3}{\ol{v}_2~b_3~b_2~b_1}~\LL{\ol{w}_1~c_1~c_2~c'~c_3~c_4}{\ol{w}_2~c_4~c_3~c'~c_2~c_1}~\LL{z}{a}},$$
		where $b_1,b_2,b_3$ compose our $3$-block, $c_1,c_2,c_3,c_4,c'$ compose our $k$-block ($c'$ exists iff $k=5$) and $\ol{v}_1,\ol{v}_2$ form $m$ consecutive $2$-blocks, for some $m\geq 0$. By performing Move \ref{Move-3217654-to-3215476} on letters $a,b_1,b_2,b_3,c_1,c_2,c_3,c_4,z$, we arrive at pair $\pp'$ of the form
		      $$ \pp'= \mtrx{\LL{a}{z}~\LL{b_3~b_2~\ol{v}_1~b_1}{b_1~b_2~\ol{v}_2~b_3}~\LL{\ol{w}_1~c_1~c_4~c'~c_3~c_2}{\ol{w}_2~c_4~c_1~c'~c_2~c_3}~\LL{z}{a}}.$$
		If $\ol{v}_1,\ol{v}_2$ are empty, then this is now a permutation of the form in Case 2A (i.e. Type $<3>$) or Case 2C. If $\ol{v}_1,\ol{v}_2$ are not empty, let $\ol{v}_1 = \ol{v}_1'd_1d_2$ and $\ol{v}_2 = \ol{v}_2'd_2d_1$. We may then focus on altering $\pp'$ to the left of $\ol{w}_1$ \& $\ol{w}_2$ by expressing $\pp'$ as
		    $$ \pp' = \mtrx{\LL{a}{z}~\LL{b_3~b_2~\ol{v}_1'~d_1~d_2~b_1}{b_1~b_2~\ol{v}_2'~d_2~d_1~b_3}~ \LL{z}{a}}.$$
		By performing switch moves $\{b_2,d_1\},\{b_1,d_2\}$, we arrive at
		    $$ \pp'' = \mtrx{\LL{a}{z}~\LL{b_1~\ol{v}_1'~d_1}{d_1~b_1~\ol{v}_2'}~\LL{b_3~b_2~d_2}{d_2~b_2~b_3}~ \LL{z}{a}}.$$
		As with $\pp'$, if $\ol{v}_1'$ \& $\ol{v}_2'$ are empty, then $\pp''$ is of the form in Case 2A. If not, we will continue by Corollary \ref{CorMove-abU-bUa}. This pair is either Case 2A (i.e. Type $<3>$) or of the form in Case 2C.

		Now suppose that we cannot place our $k$-block to the right of a $3$-block. In this instance, all $3$-blocks must belong to the rightmost chain. Choose again the leftmost such $3$-block.

		Suppose first that the $k$-block appears to the immediate left of this $3$-block. We may then proceed according to the value of $k$. If $k=4$, we may express this chain in $\pp$ as
		    $$ \pp = \mtrx{\LL{a}{z}~\LL{\ol{u}_1~b_1~b_2~b_3~b_4}{\ol{u}_2~b_4~b_3~b_2~b_1}~\LL{c_1~c_2~c_3}{c_3~c_2~c_1}~\LL{z}{a}},$$
		where $\ol{u}_1,\ol{u}_2$ form a chain of consecutive $2$-blocks. By performing Move \ref{Move-4321765-to-2143765}, we arrive at pair
		    $$ \pp' = \mtrx{\LL{a}{z}~\LL{b_3~\ol{u}_1~b_1}{b_1~b_3~\ol{u}_2}~\LL{c_1~b_4}{b_4~c_1}~\LL{c_3~c_2~b_2}{b_2~c_2~c_3}~\LL{z}{a}}.$$
		By performing the moves described in Corollary \ref{CorMove-abU-bUa} to the chain formed by letters $b_1$, $b_3$ and those in $\ol{u}_i$, we arrive at a pair such that this chain is now consecutive $2$-blocks with a $3$-block at the end. If $k=5$, we may express this chain in $\pp$ as
		    $$ \pp = \mtrx{\LL{a}{z}~\LL{\ol{u}_1~b_1~b_2~b_3~b_4~b_5}{\ol{u}_2~b_5~b_4~b_3~b_2~b_1}~\LL{c_1~c_2~c_3}{c_3~c_2~c_1}~\LL{z}{a}},$$
		where $\ol{u}_1,\ol{u}_2$ form a chain of consecutive $2$-blocks (if non-empty). By performing moves
		    $$ \{b_1,c_3\},\{b_4,c_2\},\{b_1,b_2\},\{b_3,b_5\},\{c_1,c_3\},$$
		we arrive at pair
		    $$ \pp' = \mtrx{\LL{a}{z}~\LL{\ol{u}_1}{\ol{u}_2}~\LL{b_1~b_5}{b_5~b_1}~\LL{b_3}{b_3}~\LL{c_1~b_4}{b_4~c_1}~\LL{c_1~c_2~c_3}{c_3~c_2~c_1}~\LL{z}{a}},$$
		which has replaced the original chain with two: one composed of $2$-blocks and one composed of a $2$-block followed by a $3$-block. We have successfully removed the $k$-block and now either have a permutation of Case 2A (i.e. Type $<3>$) or may proceed to Case 2C.

		By using Corollary \ref{CorMove-4blocks-to-the-left} and Move \ref{Move-43216587-to-21438765} as needed, we may change $\pp$ to the one in the preceding paragraph if there are two or more consecutive $2$-blocks to the immediate left of our $3$-block. We may then perform the moves detailed in that case to get a permutation of either Case 2A (i.e. Type $<3>$) or of Case 2C. We are therefore left considering the final case, our $3$-block appears in the rightmost chain, there is one or fewer $2$-blocks to the immediate left of our $3$-block, and by Corollary \ref{CorMove-4blocks-to-the-left} there is a $k$-block that is the leftmost block in its chain, $k=4$ or $k=5$.

		If there or no other blocks to the left of our $3$-block in its chain, then we may perform Move \ref{Move-43215876-to-21435876} to change our pair into either one of Case 2A (i.e. Type $<3>$) or of Case 2C. If there is a $2$ block to the immediate left of our $3$-block, then we may express our chains in $\pp$ as
		    $$ \pp = \mtrx{\LL{a}{z}~\LL{b_1~b_2~b'~b_3~b_4}{b_4~b_3~b'~b_2~b_1}~\LL{s}{s}~ \LL{c_1~c_2}{c_2~c_1}~\LL{d_1~d_2~d_3}{d_3~d_2~d_1}~\LL{z}{a}},$$
		where $b'$ exists iff $k=5$. By performing moves
		    $$ \{b_3,d_3\},\{c_1,d_2\},\{b_2,b_4\},\{b_2,c_1\},\{b_1,d_2\},\{b_3,c_2\},\{d_1,s\},\{b_2,d_2\},$$
		we arrive at pair
		    $$ \pp' = \mtrx{\LL{a}{z}~\LL{b_1~b_4}{b_4~b_1}~\LL{b'}{b'}~\LL{b_3~b_2}{b_2~b_3}~\LL{s}{s}~\LL{d_1~c_1}{c_1~d_1}~\LL{d_3~d_2~c_2}{c_2~d_2~d_3}~\LL{z}{a}}$$
		which has replaced the $k$-block with either two $2$-blocks when $k=4$ or two chains each of a $2$-block if $k=5$. This pair is now either Type $<3>$ or of the form for Case 2C.
		
		\textbf{Case 2C:} \textit{$\pp$ contains at least one $3$-block and all blocks are size at most $3$.} Consider a chain with more than one $3$-block. This chain is composed of only $3$-blocks and $2$-blocks. Choose the leftmost two $3$-blocks and express this chain as
		   $$ \pp = \mtrx{\LL{a}{z}~\LL{\ol{u}_1~b_1~b_2~b_3}{\ol{u}_2~b_3~b_2~b_1}~\LL{\ol{v}_1~c_1~c_2~c_3}{\ol{v}_2~c_3~c_2~c_1}~\LL{z}{a}}$$
		where $\ol{u}_1,\ol{u}_2$ and $\ol{v}_1,\ol{v}_2$ each form chains composed of $2$-blocks. By performing Move \ref{Move-321654-to-213654}, we arrive at pair
		    $$ \pp' = \mtrx{\LL{a}{z}~\LL{b_2~c_3~\ol{v}_1}{\ol{v}_2~c_3~b_2}~\LL{c_1}{c_1}~\LL{b_3~c_2~\ol{u}_1~b_1}{b_1~c_2~\ol{u}_2~b_3}~\LL{z}{a}}.$$
		By applying Corollary \ref{CorMove-abU-Uba} to the left chain, we may change it into one composed of $2$-blocks and (possibly) a $4$-block. By applying Corollary \ref{CorMove-abUc-cbUa} to the right chain, we change it into one composed of consecutive $2$-blocks followed by a $3$-block. We repeat this process until every chain contains at most one $3$-block.
		
		This process has possibly added $4$-blocks to the pair. If there is an even number of $4$-blocks, we may use Move \ref{Move-43218765-to-21436587} repeatedly to remove them, leaving a pair of Type $<3>$. If there is an odd number of $4$-blocks, me may use Move \ref{Move-43218765-to-21436587} to remove all but one of them. This word can be handled by Case 2B. While this seems circular, note that there will be at most one $3$-block in each chain of our pair. Therefore, the methods in Case 2B are guaranteed to reduce our pair to one of Type $<3>$ and \textit{not} of Case 2C again.
    \end{proof}

	\begin{thm}\label{ThmClassTypeSufficient}
	    Let $\pp,\qq\in\irr{\AAA}$ be pairs with the same Type. If $\Pof{\pp} = \Pof{\qq}$ and $\Mof{\pp}=\Mof{\qq}$, then $\ClassNon{\pp} = \ClassNon{\qq}$
	\end{thm}

	\begin{proof}
	    We proceed by the common Type of both $\pp$ and $\qq$. In general, we will produce very specific pair forms that depend on the remaining data, $\Mof{\pp}$ and $\Pof{\pp}$. Because $\Pof{\pp}=\Pof{\qq}$ and $\Mof{\pp}=\Mof{\qq}$, this will show that $\pp$ and $\qq$ will be connected to the same pair by switch moves (up to renaming), so they belong to the same one-row Rauzy Class.

	    \textbf{Type $<1>$.} Consider first $\pp$. We may express $\pp$ as
		  $$ \pp = \mtrx{\LL{a}{z}~\LL{\ol{u}}{\ol{u}}~\LL{\ol{v}_1}{\ol{v}_2}~\LL{\ol{w}}{\ol{w}}~\LL{z}{a}}$$
	    where $\ol{u}$ and $\ol{w}$ are each possibly empty words, and either $\bmtrx{\ol{v}_1\\\ol{v}_2}$ forms a $k$-block, $k\geq 2$, or $\ol{v}_1$ and $\ol{v}_2$ are empty. We note the following:
	      \begin{itemize}
		  \item $\ol{v}_1$ and $\ol{v}_2$ are non-empty if and only if $k+1$ appears as a value in $\Pof{\pp}$ exactly once. All other values in $\Pof{\pp}$ are $1$.
		  \item $\ol{w}$ is empty if and only if $\Mof{\pp} > 2$.
	      \end{itemize}
	    Suppose first that $\Mof{\pp}>2$, then $\pp$ is actually of the form
		\begin{equation}\label{EqnType1-1}
		    \pp = \mtrx{\LL{a}{z}~\LL{\ol{u}}{\ol{u}}~\LL{\ol{v}_1}{\ol{v}_2}~\LL{z}{a}}
		\end{equation}
	    where $\#\ol{u}$ is the multiplicity of $1$ in $\Pof{\pp}$. Now suppose that $\Mof{\pp}=2$. If $\ol{u}$ is non-empty, let $b$ be its last letter and let $c$ be the last letter of $\ol{w}$. By performing a $\{b,c\}$-switch, we arrive at pair
	      \begin{equation}\label{EqnType1-2}
		\pp' = \mtrx{\LL{a}{z}~\LL{\ol{v}_1}{\ol{v}_2}~\LL{\ol{w}~\ol{u}}{\ol{w}~\ol{u}}~\LL{z}{a}}.
	      \end{equation}
	    Let $\ol{w}' = \ol{w}\ol{u}$ and note that $\#\ol{w}'$ is again the number of times $\Pof{\pp}$ contains $1$.
	
	    By the same arguments above, $\qq$ belongs to the same class as a pair of the form in either \eqref{EqnType1-1} or \eqref{EqnType1-2}. This will be the same form as for $\pp$, so they all belong to the same non-labeled Rauzy Class.

	    \textbf{Type $<2>$.} We again begin with $\pp$. We know that $\pp$ is composed of chains
		$$ \textbf{C}_1,\dots,\textbf{C}_k$$
	    that are each composed of $n_i$ consecutive $2$-blocks, $n_i\geq 0$, and separated each by exactly one $1$-block. The number of $2$-blocks in $\mathbf{C}_k$ is given by $2n_k+2 = \Mof{\pp}$. All chains have lengths determined by the values in $\Pof{\pp}$ by the following:
	    $$ \Pof{\pp} = \{2n_1+1, 2n_2+1,\dots 2n_{k-1}+1, 2n_k+1\}.$$
	    By Lemma \ref{LemReorderingChains}, we may choose to order our chains $\mathbf{C}_1,\dots,\mathbf{C}_{k-1}$ in decreasing order of $n_i$. If we do the same with $\qq$, we again will connect $\pp$ and $\qq$ to the same pair in their mutual non-labeled Rauzy Class.

	    \textbf{Type $<3>$.} This is the most involved case. Let $\textbf{C}_1,\dots,\textbf{C}_k$ be the chain decomposition for $\pp$.

	    \textbf{Let us first assume that $\Mof{\pp}=\Mof{\qq}$ is even}. By Lemma \ref{LemReorderingChains}, we may arrange our chains so that $\textbf{C}_1,\dots,\textbf{C}_p$ each contain a $3$-block, $p<k$, while chains $\textbf{C}_{p+1},\dots,\textbf{C}_k$ are all composed of $2$-blocks. As with the case for Type $<2>$ pairs, we may arrange $\textbf{C}_{p+1},\dots,\textbf{C}_{k-1}$ in decreasing order of length. Let $\ell_i$ and $r_i$ denote the number of consecutive $2$-blocks to the left and right respectively of the $3$-block in chain $\textbf{C}_i$, $i\in\{1,\dots,p\}$. We will show that, by applying switch moves, we may arrive at a pair such that $\ell_i \geq r_i$ for all $i$ and $r_i \geq \ell_{i+1}$ for all $i<p$. In light of Lemma \ref{LemReorderingChains}, we must only show that we may
	      \begin{enumerate}
		\item[(a)] change a chain $\textbf{C}_i$ to interchange the values $\ell_i$ and $r_i$, and
		\item[(b)] change two chains $\textbf{C}_i$ and $\textbf{C}_{i+1}$ so that $\ell_i$ and $\ell_{i+1}$ are interchanged.
	      \end{enumerate}
	    We shall first show how to accomplish (a). Noting that such a chain cannot be rightmost in $\pp$, consider a chain $\textbf{C}_i$ in $\pp$ expressed as follows
		$$ \pp = \mtrx{\LL{a}{z}~\LL{\ol{u}_1}{\ol{u}_2}~\LL{b~c~d}{d~c~b}~\LL{\ol{v}_1}{\ol{v}_2}~\LL{s}{s}~\LL{z}{a}}$$
	    where $\ol{u}_1,\ol{u}_2$ form our chain of $\ell_i$ consecutive $2$-blocks and $\ol{v}_1,\ol{v}_2$ form our chain of $r_i$ consecutive $2$-blocks. By applying a $\{c,s\}$ switch, we arrive at
		$$ \pp' = \mtrx{\LL{a}{z}~\LL{\ol{v}_1~s~\ol{u}_1~a~b}{b~a~\ol{v}_2~s~\ol{u}_1}~\LL{c}{c}~\LL{z}{a}}.$$
	    Note that the restriction of $\pp'$ on $\AAA\setminus\{c\}$, which we will call $\rr$, is irreducible with $\Pof{\rr} = \{2\ell_i+2,2r_i+2\}$ and $\Mof{\rr} = 2\ell_i+3$. We may apply Theorem \ref{ThmTypes} to arrive at pair $\rr'$ which is composed of only one chain composed of $r_i$ consecutive $2$-blocks followed by a $3$-block followed by $\ell_i$ consecutive $2$-blocks. This extends to a pair in the same class as $\pp$ with $\textbf{C}_i$ replaced by this new chain.

	    We shall now show how to achieve (b). Let us express chains $\textbf{C}_i$ and $\textbf{C}_{i+1}$ as follows
	      $$ \pp = \mtrx{\LL{a}{z}~\LL{\ol{u}_1}{\ol{u}_2}~\LL{b~c~d}{d~c~b}~\LL{\ol{v}_1}{\ol{v}_2}~\LL{s}{s}~\LL{\ol{w}_1}{\ol{w}_2}~\LL{e~f~g}{g~f~e}~\LL{\ol{x}_1}{\ol{x}_2}~\LL{z}{a}},$$
	    where each lettered pair of words form chains composed of the appropriate number of $2$-blocks. By performing Move \ref{Move-3214765-moving-leftblocks}, we arrive at pair
	      $$ \pp' = \mtrx{\LL{a}{z}~\LL{g~c~\ol{w}_1~e}{e~c~\ol{w}_2~g}~\LL{\ol{v}_1}{\ol{v}_2}~\LL{s}{s}~\LL{\ol{u}_1}{\ol{u}_2}~\LL{b~f~d}{d~f~b}~\LL{\ol{x}_1}{\ol{x}_2}~\LL{z}{a}}.$$
	    By applying Corollary \ref{CorMove-abUc-cbUa} to letters $c,e,g$ and words $\ol{w}_1,\ol{w}_2$, we arrive at pair $\pp''$ with the values of $\ell_i$ and $\ell_{i+1}$ interchanged as desired.

	    By using these two procedures, we may arrive at pair $\tT$, whose form is entirely determined by $\Pof{\pp}$ and $\Mof{\pp}$. By applying this to $\qq$ as well, we have connected all three pairs to each other in a non-labeled Rauzy Class.

	    \textbf{Let us now assume that $\Mof{\pp}=\Mof{\qq}$ is odd.} In this case, we apply Lemma \ref{LemReorderingChains} to assume that chains $\textbf{C}_1,\dots,\textbf{C}_p$ are all composed of only $2$-blocks, $p<k$, and each chain $\textbf{C}_{p+1},\dots,\textbf{C}_k$ contains a $3$-block. We may again order the first $p$ chains in decreasing order of length. If we again look at values $\ell_{p+1},r_{p+1},\dots,\ell_k,r_k$, we may use procedures (a) and (b) to ensure that $\ell_{p+1} \geq r_{p+1} \geq \dots \geq \ell_k$, noting that $\Mof{\pp} = 2r_k+3$ and therefore $r_k$ is fixed. This again will be a pair $\tT$ whose one-row form is entirely determined by $\Mof{\pp}$ and $\Pof{\pp}$. By applying this to $\qq$ as well, we arrive at a pair which is just a renaming of $\tT$.

	    \textbf{Type $<4>$.} Just as in the Type $<2>$ case, let $\textbf{C}_1,\dots,\textbf{C}_k$ be the chain decomposition for $\pp$. We may then reorder the first $k-1$ in decreasing order according to size by Lemma \ref{LemReorderingChains}. There are now only two possibilities, the leftmost chain either is composed of at least two consecutive $2$-blocks or is composed of at most one $2$-block. In the first case, we may apply Move \ref{Move-43216587-to-21438765} to ensure that the leftmost block in the pair is size $4$. In the second case, the rightmost block \textit{must} contain a $4$-block. By Corollary \ref{CorMove-4blocks-to-the-left}, we may ensure that the $4$-block is the leftmost block in this chain. In either case, we have a pair of a unique form that is connected to $\pp$ by switch moves. The same may be shown for $\qq$, so they are all connected by switch moves up to renaming.

	    \textbf{Type $<5>$.} In this case, $\pp$ has chain decomposition $\textbf{C}_1,\dots,\textbf{C}_k$, where $\textbf{C}_j$ is composed one $5$-block for a unique $j$, and all other $\textbf{C}_i$ are each either composed of one $2$-block each or are empty. Note that in this case $\Mof{\pp}=2$ or $\Mof{\pp}=4$ and $\Pof{\pp}$ is just a list of values $1$ and $3$. We will want to connect $\pp$ to $\rr$ that has chain decomposition $\textbf{D}_1,\dots,\textbf{D}_k$, where the chains are in increasing order if $\Mof{\pp}=4$ and deceasing order if $\Mof{\pp}=2$. If either $\Mof{\pp} = 2$ or $\Mof{\pp} = 4$ and $j=k$, we may just use Lemma \ref{LemReorderingChains} to achieve this. If $\Mof{\pp}=4$ and $j < k$, we may use Lemma \ref{LemReorderingChains} to ensure that $j=k-1$. we may then express the two rightmost chains of $\pp$ as
	      $$ \pp = \mtrx{\LL{a}{z}~\LL{b_1~c_4~s~c_1~b_2}{b_2~c_1~s~c_4~b_1}~\LL{c'}{c'}~\LL{c_3~c_2}{c_2~c_3}~\LL{z}{a}}.$$
	    By performing Move \ref{Move-216543-to-432165} in reverse to $\pp$, we arrive at
		$$ \pp' = \mtrx{\LL{a}{z}~\LL{b_1~b_2}{b_2~b_1}~\LL{s}{s}~\LL{c_1~c_2~c'~c_3~c_4}{c_4~c_3~c'~c_2~c_1}~\LL{z}{a}}.$$
	    We may now apply Lemma \ref{LemReorderingChains} to $\pp'$ to arrive at our desired $\rr$. The same procedure may be applied to $\qq$ to arrive at a pair that is a renaming of $\rr$.
	  \end{proof}

	  \begin{thm}\label{ThmExclassTypeSufficient}
	    Let $\pp,\qq\in\irr{\AAA}$ be pairs with the same Type. If $\Pof{\pp} = \Pof{\qq}$, then $\ExtClassNon{\pp} = \ExtClassNon{\qq}$.
	  \end{thm}

	  \begin{proof}
	    Let $\pp\in\irr{\AAA}$ with $\Pof{\pp} = \{n_1,\dots,n_k\}$. Fix $j \in \{1,\dots,k\}$. Given Theorem \ref{ThmClassTypeSufficient}, it suffices to show that there is a pair $\rr \in \irr{\AAA}$ of the same Type as $\pp$ connected to $\pp$ by (inner and/or outer) switch moves with $\Mof{\rr} = n_j+1$. By Proposition \ref{PropPofPiExclassInvariant}, we would automatically have that $\Pof{\pp}=\Pof{\rr}$. Assume that $n_j+1 \neq \Mof{\pp}$.

	    If $n_j= 2m+1$ is odd, then there is a chain $\textbf{C}$ that is composed of either $m$ consecutive $2$-blocks, one $4$-block followed by $m-2$ consecutive $2$-blocks. Let us denote by $s$ the letter that composes the $1$-block to the right of $\textbf{C}$. Then we may express $\pp$ as
	      $$ \pp = \mtrx{\LL{a}{z}~\LL{\ol{u}_1}{\ol{u}_2}~\LL{s}{s}~\LL{\ol{v}_1}{\ol{v}_2}~\LL{z}{a}}.$$
	    By performing an $\{s,z\}$-switch, we arrive at
	      $$ \rr = \mtrx{\LL{a}{s}~\LL{\ol{v}_1}{\ol{v}_2}~\LL{z}{z}~\LL{\ol{u}_1}{\ol{u}_2}~\LL{s}{a}}.$$
	    Indeed $\textbf{C}$ is the rightmost chain of $\rr$, so $\Mof{\rr} = n_j+1$. Moreover, $\rr$ matches $\pp$ in Type, and $\Pof{\rr} = \Pof{\pp}$.

	    If $n_j = 2m$ is even, then there is a chain $\textbf{C}$ in $\pp$ that is composed of a $3$-block either preceded by or followed by $m-1$ consecutive $2$-blocks. Note that if $n_j+1\neq \Mof{\pp}$, then either $\textbf{C}$ is not the rightmost chain or is the rightmost chain with the $m-1$ consecutive $2$-blocks preceding the $3$-block.

	    If $\textbf{C}$ is not the rightmost chain, then we may apply procedure (a) from the proof of Theorem \ref{ThmClassTypeSufficient} if necessary to ensure that the $m-1$ consecutive $2$-blocks follow the $3$-block. By calling $s$ the letter of the $1$-block, we may again perform an $\{s,z\}$-switch to arrive at $\rr$.

	    Now assume that $\textbf{C}$ is the rightmost chain and the $m-1$ consecutive $2$-blocks is to the left of its $3$-block. If there is another chain with a $3$-block, then we apply procedures (a) and (b) to $\pp$ to apply the procedure in the previous paragraph. If there are no other chains containing a $3$-block, then we may express $\textbf{C}$ in $\pp$ as
	      $$ \pp = \mtrx{\LL{a}{z}~\LL{\ol{u}_1}{\ol{u}_2}~\LL{b~c~d}{d~c~b}~\LL{\ol{v}_1}{\ol{v}_2}~\LL{z}{a}}.$$
	    By performing a $\{b,z\}$-switch, we arrive at
	      $$ \pp' = \mtrx{\LL{a}{b}~\LL{c~d~\ol{v}_1~z~\ol{u}_1}{\ol{v}_2~z~\ol{u}_2~d~c}~\LL{b}{a}}.$$
	    We may apply Corollary \ref{CorMove-abUcV-UcVba} to arrive at $\rr$ which has has preserved every chain of $\pp$ except $\textbf{C}$ which has switched the number of $2$-blocks on each side of the $3$-block. Therefore $\Mof{\rr} = n_j+1$ as desired, and $\rr$ will have the same Type as $\pp$.
	  \end{proof}

	\subsection{A Class May Contain Pairs of at Most One Type}

			    Theorem \ref{ThmExclassUniqueType} tells us that the Type associated to an Extended Rauzy Class, and therefore a Rauzy Class, is unique. Lemma \ref{LemHyp} shows that any Extended Rauzy Class that contains a Type $<1>$ pair can only contain pairs of Type $<1>$. The invariant $\Pof{\pp}$ separates Type $<3>$ from the remaining Types, and Corollary \ref{CorSpin} separates Type $<2>$ from Types $<4>$ and $<5>$.

		    \begin{thm}\label{ThmExclassUniqueType}
				Let $\pp\in\irr{\AAA}$ be a pair of type $<x>$. If $\qq \in \ExtClassLab{\pp}$ is a pair of Type $<y>$, then $x=y$.
		    \end{thm}

		    \begin{proof}
				Lemma \ref{LemHyp} shows that an Extended Rauzy Class that contains a pair of Type $<1>$ can only contain pairs of Type $<1>$. We must therefore consider only the other Types.

				Consider $\pp$ with a Type. If $\pp$ is Type $<3>$, then $\Pof{\pp}$ from Definition \ref{DefTildeSofPiAndPofPi} contains even numbers. If $\pp$ is Type $<2>$, $<4>$ or $<5>$, then $\Pof{\pp}$ only contains odd numbers. So by Proposition \ref{PropPofPiExclassInvariant}, $\ExtClassLab{\pp}$ may not then contain a pair of Type $<3>$ and another pair of Type $<2>$, $<4>$ or $<5>$.
				
				If $\pp$ is Type $<5>$, then $\Pof{\pp}$ only has $3$'s and $1$'s with possible repetition. If $\pp$ is Type $<4>$, then $\Pof{\pp}$ contains an odd number that is greater than or equal to $5$. So $\ExtClassLab{\pp}$ may not contain a pair of Type $<4>$ and another of Type $<5>$.
				
				By Corollary \ref{CorSpin}, $\ExtClassLab{\pp}$ can not contain a pair of Type $<2>$ and another of either Type $<4>$ or $<5>$.

				We have concluded that the Type of any pair in $\ExtClassLab{\pp}$ must be the same as the Type of $\pp$.
		    \end{proof}

		    \begin{lem}\label{LemHyp}
				Let $\Sigma \subset \std{\AAA}$ be the set of all standard pairs $\pp$ of the form
				  $$ \mtrx{\LL{a}{z}~\LL{\ol{w}_0}{\ol{w}_0}~\LL{\ol{w}_1~\ol{w}_2 \cdots \ol{w}_{d-1}~\ol{w}_d}{\ol{w}_d~\ol{w}_{d-1} \cdots \ol{w}_2~\ol{w}_1}~\LL{\ol{w}_{d+1}}{\ol{w}_{d+1}}~\LL{z}{a}}$$
				where $\ol{w}_0$ and $\ol{w}_{d+1}$ are possibly empty words and all other words are non-empty. Then $\Sigma$ is closed under both switch and extended switch moves. If $\pp\in\Sigma$ has a Type, then it is Type $< 1 >$.
		    \end{lem}

		    \begin{proof}
				The final statement follows from the definitions of $\Sigma$ and Type. Consider $\pp\in\Sigma$ with letters $a,z$ and words $\ol{w}_0,\dots,\ol{w}_{d+1}$ as in the definition above. We will verify that the resulting pair $\pp'$ from either an inner or outer switch move also belongs to $\Sigma$. Suppose we performed a $\{b,c\}$-switch move to arrive at $\pp'$. There are three possible cases.
	
				\textit{Case 1A:} $b$ and $c$ belong the the same word $\ol{w}_i$. If $i=0$, then $\pp'$ is identical to $\pp$ except its first word $\ol{w}_0$ is a reordering of the letters in $\ol{w}_0$. If $i >0$, then let $\ol{w}_i = \ol{u}_1 b \ol{u}_2 c \ol{u}_3$, where the words $\ol{u}_1,\ol{u}_2, \ol{u}_3$ are possibly empty. Then $\pp'$ is of the form in $\Sigma$ with words $\ol{w}'_0,\dots,\ol{w}'_{d+1}$ where
					  $$ \ol{w}'_j = \RHScase{\ol{u}_2c\ol{w}_0, & j=0, \\ \ol{u}_1 b\ol{u}_3, & j=i \\ \ol{w}_j, & \mathrm{otherwise.}}$$

				\textit{Case 2A:} $b$ belongs to $\ol{w}_0$ and $c$ belongs to $\ol{w}_i$ for some $i\in\{1,\dots,d,d+1\}$. In this case, let $\ol{w}_0 = \ol{u}_1 b\ol{u}_2$ and $\ol{w}_i = \ol{v}_1 c \ol{v}_2$. Then $\pp'\in\Sigma$ with words $\ol{w}'_0,\dots,\ol{w}'_{d+1}$ where $\ol{w}'_0 = \ol{u}_2$, $\ol{w}'_i = \ol{v}_1c\ol{u}_1b\ol{v}_2$ and $\ol{w}'_j = \ol{w}_j$ for all other $j$.

				\textit{Case 3A:} $b$ belongs to $\ol{w}_i$ where $i\in\{1,\dots,d\}$ and $c$ belongs to $\ol{w}_{d+1}$. Let $\ol{w}_i = \ol{u}_1 b \ol{u}_2$ and $\ol{w}_{d+1} = \ol{v}_1 c \ol{v}_2$. Then $\pp'\in\Sigma$ with words $\ol{w}'_j$, $j\in\{0,\dots,d+1\}$, where
			      $$ \ol{w}'_j = \RHScase{\ol{u}_2, & j=0, \\ \ol{w}_{i+j}, & 1\leq j \leq d-i, \\ \ol{v}_1c\ol{w}_0, & j=d+1-i, \\
				      \ol{w}_{i+j-d-1}, & d+1-i<j\leq d, \\ \ol{u_1}b\ol{v}_2, & j=d+1.}$$

				These cases exhaust all possible inner switch moves. Note that if $b$ belongs to some word $\ol{w}_i$ and $c$ belongs to some word $\ol{w}_j$, $i,j\in\{1,\dots,d\}$, then a $\{b,c\}$-switch is not possible. Now we shall assume that $\pp'$ is the result of an outer switch move on $\pp$. We shall call this a $\{b,c\}$-switch where $c\in\{a,z\}$ (the following all hold for either choice of $c$). There are again three cases.

				\textit{Case 1B:} $b$ belongs to $\ol{w}_0$. Let $\ol{w}_0 = \ol{u} b \ol{v}$. Then $\pp'$ belongs to $\Sigma$ with words $\ol{w}_j$, $j\in\{0,\dots,d+1\}$, where $\ol{w}'_0 = \ol{v}$, $\ol{w}'_{d+1} = \ol{w}_{d+1} c \ol{u}$ and $\ol{w}'_j = \ol{w}_j$ for all other $j$.

				\textit{Case 2B:} $b$ belongs to $\ol{w}_{d+1}$. Let $\ol{w}_{d+1} = \ol{u} b \ol{v}$. Then $\pp'$ belongs to $\Sigma$ with words $\ol{w}'_j$, $j\in\{0,\dots,d+1\}$, where $\ol{w}'_0 = \ol{v} c \ol{w}_0$, $\ol{w}'_{d+1} = \ol{u}$ and $\ol{w}'_j = \ol{w}_j$ for all other $j$.

				\textit{Case 3B:} $b$ belongs to $\ol{w}_i$ for some $i\in\{1,\dots,d\}$. Let $\ol{w}_j = \ol{u} b \ol{v}$. Then $\pp' \in\Sigma$ with words $\ol{w}'_j$, $j\in\{0,\dots,d+1\}$, where
			      $$ \ol{w}'_j = \RHScase{\ol{v}, & j=0 \\ \ol{w}_{i+j}, & 1\leq j \leq d-i \\ \ol{w}_{d+1} c \ol{w}_0, & j = d+1-i\\
					    \ol{w}_{i+j-d-1}, & d+1-i < j \leq d\\ \ol{u}, & j=d+1.}$$

				We have now that for any $\pp\in\Sigma$ and $\pp'$ connected to $\pp$ by a switch move, $\pp'\in\Sigma$ as well.
		    \end{proof}

		    \begin{lem}\label{LemArfBlockCount}
			  If piece-wise order reversing $\pp$ is composed of $N_1$ $1$-blocks, $N_2$ $2$-blocks, $N_4$ $4$-blocks and $N_5$ $5$-blocks (and no other blocks), then
				  \begin{equation}\label{EqSpinCount}
				      \ARFof{\pp} = 2^{N_1 + 2N_2 + 4 N_4 + 5 N_5 + 1} + (-1)^{N_4 + N_5} 2^{N_1+N_2 + 2 N_4 + 3 N_5}.
				   \end{equation}
		    \end{lem}

		    \begin{proof}
			    The proof may be performed by induction on each of the $N_j$'s. Given $\pp=(p_0,p_1)\in\std{\AAA}$ with $a = p^{-1}(0)$ and $z = p^{-1}(1)$ as the first and last letters of each row, let
				\begin{equation}
				      \ARFofE{\pp} = \#\{v\in \ZZ_2^\AAA: v_a = v_z\mbox{ and } \QF{\pp}(v) = 1 \},
				\end{equation}
			      and
				\begin{equation}
				      \ARFofNE{\pp} = \#\{v\in \ZZ_2^\AAA: v_a \neq v_z\mbox{ and } \QF{\pp}(v) = 1 \},
				\end{equation}
			      where $\QF{\pp}$ is the canonical quadratic form associated to $\pp$ from Definition \ref{DefCanonQuad}. Note in this case that $\ARFof{\pp} = \ARFofE{\pp} + \ARFofNE{\pp}$. By counting, we have that
				\begin{equation}\label{EqCountByBlocks}
				    \#\AAA = 2 + N_1 + 2 N_2 + 4 N_4 + 5 N_5.
				\end{equation}

			    We will outline the case for $N_2$, as the rest are very similar. Suppose $\pp$ is of the form
				$$ \pp = \mtrx{\LL{a}{z}~\mathbf{B}&\mathbf{C}&\LL{z}{a}},\mbox{ where }\mathbf{B} = \bmtrx{b & c \\ c & b},$$
			  and let $\pp' = \mtrx{\LL{a}{z}~\mathbf{C}&  \LL{z}{a}}$. Note that $\pp'$ has one fewer $2$-block than $\pp$ and the same number of $j$-blocks for all other $j$. From the definition,
				\begin{equation}\label{EqSpinProof}
				      \QF{\pp}(v) = v_b^2 + v_c^2 + (v_b+v_c)(v_a+v_z) + v_b v_c + \QF{\pp'}(v')\mod 2
				\end{equation}
			  where $v\in \ZZ_2^\AAA$ and $v'$ is the projection of $v$ onto $\ZZ_2^{\AAA\setminus\{b,c\}}$.

			    When $v_a = v_z$, \eqref{EqSpinProof} becomes $\QF{\pp}(v) = v_b^2 + v_c^2 + v_b v_c + \QF(\pp')(v')$ and so by counting
				  $$ \begin{array}{rcl}
				      \ARFofE{\pp} &= &\ARFofE{\pp'} + 3 \big(2^{\#\AAA-3} - \ARFofE{\pp'}\big)\\
					  &= &3\cdot 2^{\#\AAA-3} - 2 \cdot \ARFofE{\pp'}.
				     \end{array}$$
			    When $v_a \neq v_z$, \eqref{EqSpinProof} now becomes $\QF{\pp}(v) = v_b v_c + \QF{\pp'}(v)$ and so
				  $$ \begin{array}{rcl}
				      \ARFofNE{\pp} &= &3\cdot \ARFofNE{\pp'} + \big(2^{\#\AAA-3} - \ARFofNE{\pp'}\big)\\
					  &= & 2^{\#\AAA-3} + 2 \cdot \ARFofNE{\pp'}.
				     \end{array}.$$

			    We may obtain similar relationships between $\ARFofE{\pp}$ (resp. $\ARFofNE{\pp}$) and $\ARFofE{\pp'}$ (resp. $\ARFofNE{\pp'}$) where $\pp'$ is the result of removing a $j$-block for each other $j$. We are left with recurrence relations for $\ARFofE{\pp}$ and $\ARFofNE{\pp}$ that depend on $N_1$, $N_2$, $N_4$ and $N_5$. The final formulation (left as an exercise for the reader) is
				  $$ \ARFofE{\pp} = 2^{\#\AAA-2}\mbox{ and } \ARFofNE{\pp} = 2^{\#\AAA-2} + (-1)^{N_4+N_5} 2^{N_1 + N_2 + 2N_4 + 3N_5}.$$
			    The desired expression comes by $\ARFof{\pp} = \ARFofE{\pp} + \ARFofNE{\pp}$ and \eqref{EqCountByBlocks}.
		    \end{proof}
    
		    \begin{cor}\label{CorSpin}
				If $\pp,\qq\in\std{\AAA}$ are standard pairs, $\pp$ is Type $<2>$ and $\qq$ is either or Type $<4>$ or Type $<5>$, then $\ExtClassNon{\pp}\neq \ExtClassNon{\qq}$.
		    \end{cor}

		    \begin{proof}
				If $\Pof{\pp}\neq \Pof{\qq}$, then the result follows by Proposition \ref{PropPofPiExclassInvariant}. We therefore assume that $\Pof{\pp}= \Pof{\qq}$.

				If $\qq$ is Type $<4>$, let $N_j$ be the number of $j$-blocks that occur in $\qq$, for $j=1,2,4$. By definition, $N_4=1$. In this case, $\pp$ has $N_1$ $1$-blocks as well, has $N_2+2$ $2$-blocks and no other blocks. By Lemma \ref{LemArfBlockCount},
					$$ \ARFof{\pp} - \ARFof{\qq} = 2^{N_1 + N_2 + 3}\neq 0.$$
				Therefore $\ExtClassNon{\pp} \neq \ExtClassNon{\qq}$ by Corollary \ref{CorArfIsInvariant}.

				Likewise, if $\qq$ is Type $<5>$ and has $N_1$ $1$-blocks, $N_2$ $2$-blocks and one $5$-block, then
					$$ \ARFof{\pp} - \ARFof{\qq} = 2^{N_1 + N_2 + 4}\neq 0,$$
				and we reach the same conclusion.
		    \end{proof}

	\section{The Completeness Of The Switch Moves}\label{SecSwitches}

		    We show in this section a particularly helpful result: every standard pair is connected to every other standard pair in the same Rauzy Class by switch moves.
		    By including outer switch moves as well, we may say the same result for standard pairs in the same Extended Rauzy Class. We used these results in the proof of Lemma \ref{LemHyp} and will use them in Section \ref{SecOther}.

		    \begin{prop}\label{PropSwitchesClass}
					Distinct $\pp,\rr\in\std{\AAA}$ are in the same labeled Rauzy Class if and only if there exists a sequence of inner switch moves connecting $\pp$ and $\rr$.
		    \end{prop}

		    \begin{proof} We first note that if a series of switch moves exist, then $\pp$ and $\rr$ belong to the same Rauzy Class, as switch moves are themselves composed of Rauzy moves. Therefore, we will assume that $\ClassLab{\pp}=\ClassLab{\rr}$ and show that path of switch moves exist connecting the two.
	
				Let $a = p_0^{-1}(1)$ and $z = p_1^{-1}(1)$. We prove this by induction on $\#\AAA$.

				Consider a cycle path between $\pp$ and $\rr$ expressed as
				    $$ \pp = \pp^{(0)} \longrightarrow \pp^{(1)} \longrightarrow \dots \longrightarrow \pp^{(m-1)} \longrightarrow \pp^{(m)}=\rr,$$
				or in other words, $\pp^{(0)},\dots,\pp^{(m)}$ are the vertices of our cycle path. By possibly considering a shorter path, assume that each $\pp^{(i)}$ is distinct. If this path is an extension of a path on $\AAA'\subsetneq\AAA$, then by induction there exists a series of switch moves in $\AAA'$ that connect $\pp\RED{\AAA'}$ and $\rr\RED{\AAA'}$ which then extends to a series of switch moves in $\AAA$ connecting $\pp$ and $\rr$. Note that $\pp\RED{\AAA'}$ and $\rr\RED{\AAA'}$ are necessarily distinct.

				If such a letter does not exist, let us instead choose two letters $b_1,c_1 \in \AAA\setminus\{a,z\}$ such that for some $2\leq m_1 < m$, the path connecting $\pp^{(0)}$ to $\pp^{(m_1)}$ is an extension of a path in $\AAA^{(1)}:=\AAA\setminus\{b_1,c_1\}$, while the path from $\pp^{(0)}$ to $\pp^{(m_1+1)}$ is not. Call $\qq^{(1)} = \pp^{(0)}$ and note that $\qq^{(1)}\RED{\AAA^{(1)}}\in\std{\AAA^{(1)}}$. By switching letters, we may see that the path from $\pp^{(0)}$ to $\pp^{(m_1+1)}$ is still an extension of a switch path on $\AAA\setminus\{b_1\}$. Call $b_2=b_1$.

				Let us first assume there is another letter $c_2\neq a,z$ such that $\pp^{(m_1+1)}\RED{\AAA^{(2)}}$ is irreducible where $\AAA^{(2)}:= \AAA\setminus\{b_2,c_2\}$. In this case, we may find a path from $\pp^{(m_1+1)}$ to standard pair $\qq^{(2)}$ that is an extension of a path on $\AAA^{(2)}$. If $\qq^{(2)}\neq \qq^{(1)}$, there exists by induction a switch path from $\pp=\qq^{(1)}$ to $\qq^{(2)}$ that is an extension of a switch path on $\AAA\setminus\{b_1=b_2\}$. We define $m_2$ by the longest path starting at $\pp^{(m_1+1)}$ and ending at $\pp^{(m_2)}$ that is an extension of a path on $\AAA^{(2)}$. If $m_2=m$ and $\qq^{(2)}\neq \rr$, then there exists a switch path from $\rr =\pp^{(m)}$ to $\qq^{(2)}$ that is an extension of a switch path on $\AAA\setminus\{b_2,c_2\}$. We have therefore connected $\pp$ to $\rr$ by switch moves.

				Suppose $m_i<m$ is then defined for letters $\{b_i,c_i\}$ and $\qq^{(i)}$ is connected to $\pp^{(m_{i-1}+1)}$ by a path that is an extension of a path on $\AAA^{(i)}=\AAA\setminus\{b_i,c_i\}$. Let us continue to assume that, by naming $b_{i+1}=b_i$, we may find another letter $c_{i+1}$ such that $\pp^{(m_i+1)}\RED{\AAA^{(i+1)}} \in \irr{\AAA^{(i+1)}}$ where $\AAA^{(i+1)} = \AAA\setminus\{b_{i+1},c_{i+1}\}$. We then define $\qq^{(i+1)}$ similarly and note that $\qq^{(i)}$ and $\qq^{(i+1)}$ are connected by a switch move that is an extension from $\AAA\setminus\{b_{i+1}=b_i\}$ by induction. We iterate this process on $i$ until $m_i = m$.

				If for some $m_i$, $\pp^{(m_i+1)}$ does not admit a letter $c_{i+1}$, then by Lemma \ref{LemFarthestFromStandard}, $\pp^{(m_i+1)}\RED{\AAA\setminus\{b_i\}}$ belongs to $\STAR{\AAA\setminus\{b_i\}}$ from Definition \ref{DefSpecialForm}. Up to switching rows, we may assume that this reduction is of the form
			    $$ \pp^{(m_i+1)}\RED{\AAA\setminus\{b_i\}}= \mtrx{\LL{\dots}{\dots}~\LL{d_4~d_3~d_2~d_1~d_0}{d_2~d_5~d_0~d_3~d_1}},$$
			where the $d_i$'s appear in reverse order in the first row. Note that $d_1=c_i$, and define $\ell_0$ and $\ell_1$ to be the position of $b_i=b_{i+1}$ in the first and second row respectively. Because of our choice of $b_i$, it follows that $\ell_0,\ell_1<\#\AAA$. If $\ell_0 < \#\AAA-1$, then $0\pp^{(m_i+1)} = \pp^{(m_i+1)}$ which implies that $\pp^{(m_i+1)}$ can not be a vertex of a cycle path. So $\ell_0 = \#\AAA-1$. We now consider the value of $\ell_1$.

			If $\ell_1 = \#\AAA-1$, then we see the following:
			   $$ \pp^{(m_i+1)} = \mtrx{\LL{\dots}{\dots}~\LL{d_2~d_1~b_i~d_0}{d_0~d_3~b_i~d_1}}
		\overset{0}{\longleftrightarrow} \pp^{(m_i)} = \mtrx{\LL{\dots}{\dots}~\LL{d_2~d_1~b_i~d_0}{d_0~b_i~d_1~d_3}} $$
			and
			   $$ \pp^{(m_i+1)} = \mtrx{\LL{\dots}{\dots}~\LL{d_2~d_1~b_i~d_0}{d_0~d_3~b_i~d_1}}\overset{1}{\longleftrightarrow}
				  \pp^{(m_i+2)} = \mtrx{\LL{\dots}{\dots}~\LL{d_2~d_1~d_0~b_i}{d_0~d_3~b_i~d_1}}.$$
			However, $1\pp^{(m_i+2)} = \pp^{(m_i+2)}$, which again implies that neither $\pp^{(m_i+1)}$ nor $\pp^{(m_i+2)}$ can be vertices of our cycle path.

			Suppose $\ell_1 = \#\AAA-2$. Then
			  $$ \pp^{(m_i+1)} = \mtrx{\LL{\dots}{\dots}~\LL{d_2~d_1~b_i~d_0}{d_0~b_i~d_3~d_1}},~
			      \pp^{(m_i)} = \mtrx{\LL{\dots}{\dots}~\LL{d_2~d_1~b_i~d_0}{d_0~d_1~b_i~d_3}}$$
		  which tells us also that
			  $$ \pp^{(m_i+2)} = \mtrx{\LL{\dots}{\dots}~\LL{d_2~d_1~d_0~b_i}{d_0~b_i~d_3~d_1}} \mbox{ and }
	      \pp^{(m_i+3)} = \mtrx{\LL{\dots}{\dots}~\LL{d_2~d_1~d_0~b_i}{d_0~b_i~d_1~d_3}}.$$
		We state the following:
			  \begin{itemize}
				   \item $\pp^{(m_i)}$ is, as desired, irreducible on $\AAA^{(i)}\setminus\{b_i,c_i=d_1\}$. By Lemma \ref{LemFarthestToHyp}, we may find $\qq^{(i+\frac{1}{2})}$ that is standard and connected to $\pp^{(m_i)}$ of the form:
							$$ \qq^{(i+\frac{1}{2})} = \mtrx{\LL{d_{n-1}}{d_{n-2}}~\LL{d_{n-3}}{d_{n-4}}~\LL{\dots}{\dots} ~\LL{d_3~d_1~b_i~d_0}{d_0~d_1~b_i~d_3}~\LL{\dots}{\dots}~\LL{d_{n-4}}{d_{n-3}}~\LL{d_{n-2}}{d_{n-1}}}$$
					   if $\#\AAA$ is odd, or
							$$ \qq^{(i+\frac{1}{2})} = \mtrx{\LL{d_{n-1}}{d_{n-3}}~\LL{d_{n-2}\dots}{.\dots.~d_4} ~\LL{d_3~d_1~b_i~d_0~d_2}{d_2~d_0~d_1~b_i~d_3}~\LL{d_4~.\dots.}{\dots d_{n-2}}~\LL{d_{n-3}}{d_{n-1}}}$$
					   if $\#\AAA$ is even. In either case, if distinct, $\qq^{(i)}$ and $\qq^{(i+\frac{1}{2})}$ are connected by switch moves over $\AAA\setminus\{b_i,c_i\}$ by induction.
					   
				   \item $\pp^{(m_i+3)}\RED{\AAA^{(i+1)}}\in\irr{\AAA^{(i+1)}}$ for $\AAA^{(i+1)}:=\AAA\setminus\{d_0,d_1\}$ and is in fact of the form in Definition \ref{DefSpecialForm}. Let $\qq^{(i+1)}$ be the standard pair connected to $\pp^{(m_i+3)}$ by a path extended from $\AAA^{(i+1)}$. Again by Lemma \ref{LemFarthestToHyp}, we see that
					      $$ \qq^{(i+1)} = \mtrx{\LL{d_{n-1}~d_{n-3}}{d_{n-2}~d_{n-4}}~\LL{\dots}{\dots} ~\LL{d_3~d_1~d_0~b_i~d_2}{d_2~d_0~b_i~d_1~d_3}~\LL{\dots}{\dots}~\LL{d_{n-4}~d_{n-2}}{d_{n-3}~d_{n-1}}}$$
					    if $\#\AAA$ is odd, or 
					      $$ \qq^{(i+1)} = \mtrx{\LL{d_{n-1}~d_{n-2}}{d_{n-3}~d_{n-5}}~\LL{d_{n-4}\dots}{.\dots.~d_2} ~\LL{d_3~d_1~d_0~b_i}{d_0~b_i~d_1~d_3}~\LL{d_2~.\dots.}{\dots d_{n-4}}~\LL{d_{n-5}~d_{n-3}}{d_{n-2}~d_{n-1}}}$$
					    if $\#\AAA$ is even.
					    
				   \item In each case, $\qq^{(i+\frac{1}{2})}$ is connected to $\qq^{(i+1)}$ by switch moves $\{b_i,d_1\},\{b_i,d_0\}$.
			  \end{itemize}
			If $\ell_1=\#\AAA-3$, we repeat this argument, noting that the order of pairs will be in reverse and using the substitutions $b_i' = d_0$, $d_0' = b_i$ and $\pp^{(m_i+1)'} = \pp^{(m_i+2)}$.

			If $\ell_1 < \#\AAA-3$, then
				  $$ \pp^{(m_i)} = \mtrx{\LL{\dots}{\dots}~\LL{d_2~d_1~b_i~d_0}{d_5~d_0~d_1~d_3}}\mbox{ and }
				      \pp^{(m_i+2)} = \mtrx{\LL{\dots}{\dots}~\LL{d_2~d_1~d_0~b_i}{d_5~d_0~d_3~d_1}}.$$
			We may verify the following:
			  \begin{itemize}
				   \item $\pp^{(m_i)}\RED{\AAA^{(i)}}$ and $\pp^{(m_i)}\RED{\AAA\setminus\{b_i,d_2\}}$ are irreducible.
				   \item $\pp^{(m_i+2)}\RED{\AAA^{(i+1)}}$ is irreducible, where $\AAA^{(i+1)}:= \AAA\setminus\{d_0,d_2\}$.
			  \end{itemize}
			Let $\qq^{(i+1)}$ be a standard pair connected to $\pp^{(m_i+2)}$ by a path extended from $\AAA^{(i+1)}$. Also, let $\qq^{(i+\frac{1}{2})}$ be a standard pair connected to $\pp^{(m_i)}$ by a path extended from $\AAA\setminus\{b_i,d_2\}$. If distinct, $\qq^{(i)}$ and $\qq^{(i+\frac{1}{2})}$ are connected by switch moves by induction, as they are connected by a path extended from $\AAA\setminus\{b_i\}$. Likewise, $\qq^{(i+\frac{1}{2})}$ and $\qq^{(i+1)}$ are connected by a path extended from $\AAA\setminus\{d_2\}$ and are therefore also connected by a switch path.

			There will exist $k$ such that $m_1 < \dots < m_k=m$ and either $\qq^{(m)} = \rr$ or $\qq^{(m)}$ will be connected to $\rr$ by switch moves that are extended from $\AAA^{(m)}$. We have therefore shown a sequence of switch moves starting at $\pp = \qq^{(1)}$, traveling through each $\qq^{(i)}$ in order and ending at $\rr$ as desired.
	    \end{proof}

	    \begin{prop}\label{PropSwitchesExclass}
			Distinct standard pairs $\pp$ and $\rr$ are in the same labeled Extended Rauzy Class if and only if there exists a sequence of inner and/or outer switch moves connecting $\pp$ and $\rr$.
	    \end{prop}

	    \begin{proof}
			By Proposition \ref{PropSwitchesClass}, we must only show that standard pairs in an Extended Rauzy Class are connected by switch moves (of either type) if they belong to distinct Rauzy Classes. We will further simplify the proof by considering any pair $\sS$ and $\sS' = \tilde{1} \sS$. If we show that particular standard pairs $\qq \in \ClassLab{\sS}$ and $\qq' \in \ClassLab{\sS'}$ are connected by switch moves, the proof is complete. An analogous argument applies if instead $\sS' = \tilde{0} \sS$. We proceed by cases.

			If $\sS$ is itself standard, then let $\qq = \sS$. There must exist $k>0$ such that $\qq' = 0^k\sS'$ is standard. This move described between $\qq$ and $\qq'$ is an outer switch move.

			Now suppose $\sS$ is within an $\eps$-cycle of standard pair $\qq$, $\eps\in\{0,1\}$. If $\eps = 0$, then it follows immediately that $\sS''= \tilde{1} \qq$ belongs to $\RRR(\sS')$. As in the previous case, there exists $\qq'$ that is standard in the $0$-cycle of $\sS''$, and the path described connecting $\qq$ to $\sS''$ to $\qq'$ is a switch move. Now suppose that $\eps = 1$. Let us express $\sS$ and $\sS'$ in the following way:
			    $$ \xymatrix{ \sS = \mtrx{\LL{a}{z}~\LL{c~\ol{w}_1}{\ol{w}_3}~\LL{z}{c}~\LL{\ol{w}_2}{\ol{w}_4}~\LL{b}{a}} \ar[r]^{\tilde{1}} 
					  & \mtrx{\LL{c}{z}~\LL{\ol{w}_1~a}{\ol{w}_3}~\LL{z}{c}~\LL{\ol{w}_2}{\ol{w}_4}~\LL{b}{a}} = \sS',}$$
			where $\ol{w}_1,\ol{w}_2,\ol{w}_3,\ol{w}_4$ are all possibly empty words. We then see that $\qq = 1^m\sS$, where $m = 1+\#\ol{w}_2$. Let $\sS'' = \tilde{1}^n \qq$, where $n = 2 + \#\ol{w}_2$. We may also verify that $\sS'' = 1^m\sS'$, so $\sS''\in\ClassLab{\sS'}$. Let $\qq' = 0^p\sS''$, where $p = 1+\#\ol{w}_4$. All of these moves are expressed in Figure \ref{FigProofExswitch1}. It follows that $\qq$ and $\qq'$ are connected by switch moves.

			We now consider $\sS$ that is more than one cycle length away from a standard pair. When moving to a standard pair by the proof of Proposition \ref{PropStandardInClass}, we choose an $\eps$ and $k>0$ to arrive at $\tT = \eps^k \sS$ that is one cycle length closer to a standard pair. If $\eps = 0$, it follows immediately that $\tT' = \tilde{1} \tT = 0^k \sS'$. If we replace $\sS$ and $\sS'$ with $\tT$ and $\tT'$, we may assume that we are closer to a standard pair. If instead $\eps = 1$, we may express $\sS$ and $\sS'$ as follows:
		    	$$ \xymatrix{ \sS = \mtrx{\LL{a}{c}~\LL{b~\ol{w}_1~c~\ol{w}_2~d~\ol{w}_3~e~\ol{w}_4}{\ol{w}_5}~\LL{f}{d}} \ar[r]^{\tilde{1}}
					  & \mtrx{\LL{b}{c}~\LL{\ol{w}_1~a~c~\ol{w}_2~d~\ol{w}_3~e~\ol{w}_4}{\ol{w}_5}~\LL{f}{d}} = \sS',}$$
			where $e = t_0^{-1}(\#\AAA)$ is the last letter on the top row of $\tt$. It follows that $k = 1+ \#\ol{w}_4$. We may verify that $\tT'= \tilde{1}\tT$ is also equal to $1^k\sS'$. So again by replacing $\sS$ and $\sS'$ with $\tT$ and $\tT'$ respectively, we may assume that $\sS$ is always at most one cycle away from a standard pair. The argument then finishes by using the previous cases.
	
			\begin{figure}[t]
				  $$ \xymatrix{ \sS = \mtrx{\LL{a}{z}~\LL{c~\ol{w}_1}{\ol{w}_3}~\LL{z}{c}~\LL{\ol{w}_2}{\ol{w}_4}~\LL{b}{a}} \ar[r]^{\tilde{1}} \ar[d]_{1^m}
					  & \mtrx{\LL{c}{z}~\LL{\ol{w}_1~a}{\ol{w}_3}~\LL{z}{c}~\LL{\ol{w}_2}{\ol{w}_4}~\LL{b}{a}} = \sS'\ar[d]^{1^m} \\
				      \qq = \mtrx{\LL{a}{z}~\LL{\ol{w}_2}{\ol{w}_3}~\LL{b}{c}~\LL{c~\ol{w}_1}{\ol{w}_4}~\LL{z}{a}} \ar[r]^{\tilde{1}^n}
						  & \mtrx{\LL{c}{z}~\LL{\ol{w}_1}{\ol{w}_3}~\LL{a}{c}~\LL{\ol{w}_2~b}{\ol{w}_4}~\LL{z}{a}}=\sS'' \ar[d]^{0^p} \\
					  & \mtrx{\LL{c}{z}~\LL{\ol{w}_1}{\ol{w}_4}~\LL{a}{a}~\LL{\ol{w}_2~b}{\ol{w}_3}~\LL{z}{c}}=\qq'}$$
			  \caption{Showing the moves for the second case in the proof of Proposition \ref{PropSwitchesExclass}.}\label{FigProofExswitch1}
			\end{figure}

		    \end{proof}

	\section{More Applications of the Switch Move}\label{SecOther}

	\subsection{Hyperelliptic Rauzy Classes}

	  While already well-known since conception of Rauzy Classes (see \cite{cRau1979}), we provide a ``one line proof" of a structure question regarding \term{hyperelliptic classes}, which for the purpose of this paper will mean a Rauzy Class that contains an order reversing pair.

	  \begin{prop}\label{PropOrderReversingOnlyStandard}
	      If $\pp\in\irr{\AAA}$ is order reversing, or
			$$ p_0(b) + p_1(b) = \#\AAA+1$$
	      for all $b\in\AAA$, then $\pp$ is the only standard pair in $\ClassLab{\pp}$.
	  \end{prop}

	  \begin{proof}
    	  By definition, there are no (regular) switch moves possible for $\pp$. By Proposition \ref{PropSwitchesClass}, this implies that no other standard pairs exist in $\ClassLab{\pp}$.
	  \end{proof}
  
	\subsection{Good Pairs}
		In \cite[Lemma 20]{cKonZor2003} and \cite[Lemma 3.7]{cAvViana2007}, the respective authors proved conditions for which a ``good'' pair exists within any Extended Rauzy Class, \cite{cKonZor2003}, or in any Rauzy Class, \cite{cAvViana2007}. We will prove a stronger version of either lemma by using switch moves. In order to do so, we must define some properties of pairs.
		  \begin{defn} \label{DefGoodDegenerate}
		    Let $\pp \in\std{\AAA}$. We call $\pp=(p_0,p_1)$ \term{degenerate$^*$} if for some $c\in \AAA$,
				$$ p_0(c) = p_1(c) = \#\AAA-1.$$
		    We call $\pp$ \term{good} if the restriction of $\pp$ to positions $2$,$3$,$\dots$,$\#\AAA-1$ is irreducible.
		  \end{defn}

		  \begin{prop}
		    Every Rauzy Class on at least $4$ letters either contains a good pair or a degenerate$^*$ pair.
		  \end{prop}

		  \begin{proof}
		    By Theorem \ref{ThmPWOrderReversing}, we may consider a pair $\pp$ in the class that is piece-wise order reversing. Suppose $\pp$ is not degenerate$^*$. If $\pp$ is actually order reversing, then so is $\pp$ restricted to $\AAA \setminus \{a,z\}$, for $a = p_0^{-1}(1)$, $z = p_1^{-1}(1)$. This restricted pair is irreducible, so $\pp$ is good. Otherwise, we may write $\pp$ as
			$$ \pp = \mtrx{\LL{a}{z} ~ \LL{c}{d} ~ \LL{\ol{u}_0}{\ol{u}_1} ~ \LL{d}{c} ~ \LL{\ol{v}_0}{\ol{v}_1} ~ \LL{e}{f} ~ \LL{\ol{w}_0}{\ol{w}_1} ~ \LL{f}{e} ~ \LL{z}{a}}$$
		    where $\ol{u}_\eps$, $\ol{v}_\eps$ and $\ol{w}_\eps$ are (possibly empty) words on disjoint sub-alphabets $\AAA_u$, $\AAA_v$ and $\AAA_w$ respectively. Let $\pp'$ be the result of a $\{c,e\}$-switch on $\pp$. By Lemma \ref{LemIrreducibleExtensions},
			  $$ \pp' = \mtrx{\LL{a}{z} ~ \LL{\ol{u}_0}{\ol{v}_1} ~ \LL{d}{f} ~ \LL{\ol{v}_0}{\ol{w}_1} ~ \LL{e}{e} ~ \LL{c}{d} ~ \LL{\ol{w}_0}{\ol{u}_1} ~ \LL{f}{c} ~ \LL{z}{a}}$$
		    is good.
		  \end{proof}

	\section{Further Work}

	While not included in this work, these methods extend to fully reprove the results in \cite{cBoi2013}. That paper provided a formula for the cardinality of a labeled Rauzy Class compared to its corresponding non-labeled Class.
	An open question presented in that paper has a solution in \cite{cFick2014b}, a conceptual sequel to our work here.

	The combinatorial structure of Generalized Rauzy Classes remain open (see \cite{cBoiLan2007} for work in this direction). For example, there is not currently a ``convenient" Generalized Pair that is known to exist in every Class, analogous to standard pairs in Rauzy Classes. In \cite{cFickThesis}, a form of self-inverse pairs were found to exist in \textit{some} classes. In \cite{cZor2008}, Zorich provided explicit Generalized Forms in each class, but these forms were designed with surfaces in mind rather than considering Rauzy Moves.

\appendix

\section{Switch Moves Concerning \texorpdfstring{$2$}{2}-blocks And \texorpdfstring{$4/5$}{4/5}-blocks}

\begin{move} \label{Move-2431-to-2143}
    The pairs
      $$ \pp = \mtrx{\LL{a}{z} ~ \LL{b_1~b_2~c_1~c_2}{b_2~c_2~c_1~b_1} ~ \LL{z}{a}} \mbox{, and }
	  \qq = \mtrx{\LL{a}{z}~\LL{c_2~b_1}{b_1~c_2}~\LL{c_1~b_2}{b_2~c_1}~\LL{z}{a}}$$
     are connected by switch moves $\{b_2,c_2\}, \{b_1,c_1\}$.
\end{move}

\begin{cor}\label{CorMove-abU-bUa}
  Suppose $\pp$ is a pair of the form
      $$ \pp = \mtrx{\LL{a}{z}~\LL{b_1~b_2~\ol{u}_1}{b_2~\ol{u}_2~b_1}~\LL{z}{a}}$$
  where $\ol{u}_1,\ol{u}_2$ form a chain composed of $m$ consecutive $2$-blocks, $m\geq 0$. Then there exists an piece-wise order reversing pair $\qq$ connected to $\pp$ by switch moves that is composed of $m+1$ consecutive $2$-blocks.
\end{cor}

\begin{proof} Induction on $m$. If $m=0$, then $\qq = \pp$.

    If $m>0$, then let $\ol{u}_1=\ol{u}_1'c_1c_2$ and $\ol{u}_2 = \ol{u}'_2 c_2c_1$. Here $\ol{u}_1',\ol{u}_2'$ form a chain of $m-1$ consecutive $2$-blocks. Pair $\pp$ may then be expressed as
	$$ \pp = \mtrx{\LL{a}{z} ~ \LL{b_1~b_2~\ol{u}_1'~c_1~c_2}{b_2~\ol{u}_2'~c_2~c_1~b_1} ~ \LL{z}{a}}.$$
    We may apply Move \ref{Move-2431-to-2143} to arrive at
	$$ \pp' = \mtrx{\LL{a}{z}~\LL{c_2~b_1~\ol{u}_1'}{b_1~\ol{u}_2'~c_2}~\LL{c_1~b_2}{b_2~c_1}~\LL{z}{a}}.$$
    By induction, we may act on the letters in $\ol{u}_i'$ as well as $c_2$ and $b_1$ to arrive at pair $\qq$. The letters $b_2$ and $c_1$ form a $2$-block, and the other letters form $m$ consecutive $2$-blocks.
\end{proof}

\begin{move} \label{Move-436521-to-214365}
   The pairs
      $$\pp = \mtrx{\LL{a}{z}~\LL{b_1~b_2~b_3~b_4~b_5~b_6}{b_4~b_3~b_6~b_5~b_2~b1}~\LL{z}{a}} \mbox{ and }
	  \qq = \mtrx{\LL{a}{z}~\LL{b_1~b_4}{b_4~b_1}~\LL{b_3~b_6}{b_6~b_3}~\LL{b_5~b_2}{b_5~b_2}~\LL{z}{a}}$$
   are connected by the following switch moves: $\{b_3,b_6\},\{b_1,b_5\},\{b_2,b_3\},\{b_4,b_6\}.$
\end{move}

\begin{cor} \label{CorMove-abU-Uba} Suppose $\pp$ is a pair of the form
    $$ \pp = \mtrx{\LL{a}{z}~\LL{b_1~b_2~\ol{u}_1}{\ol{u}_2~b_2~b_1}~\LL{z}{a}}$$
    where $\ol{u}_1,\ol{u}_2$ form a chain composed of $m$ consecutive $2$-blocks, $m\geq 0$.
    Then there exists $\qq$ connected by switch moves to $\pp$ such that $\qq$ is piecewise-order reversing and is composed of
	\begin{itemize}
		\item $m+1$ consecutive $2$-blocks if $m$ is even or
		\item one $4$-block followed by $m-1$ consecutive $2$-blocks if $m$ is odd.
	\end{itemize}
\end{cor}

\begin{proof}
  This is a proof by induction on $m$.

  If $m=0$ or $m=1$, $\pp=\qq$ and the result is trivial.

  Now consider the case for $m\geq 2$. We may then express $\ol{u}_1$ as $\ol{u}'_1 c_1 c_2 c_3 c_4$ and $\ol{u}_2$ as $\ol{u}'_2 c_2 c_1 c_4 c_3$ where $\ol{u}_1'$ and $\ol{u}_2'$ are still of the form in the statement of the corollary, but with $m'=m-2$. We may then apply Move \ref{Move-436521-to-214365} to
      $$ \pp = \mtrx{\LL{a}{z}~\LL{b_1~b_2~\ol{u}_1'~c_1~c_2~c_3~c_4}{\ol{u}_2'~c_2~c_1~c_4~c_3~b_2~b_1}~\LL{z}{a}}$$
  to arrive at
      $$ \pp' = \mtrx{\LL{a}{z}~\LL{b_1~c_2~\ol{u}_1'}{\ol{u}_2'~c_2~b_1}~\LL{c_1~c_4}{c_4~c_1}~\LL{c_3~b_2}{b_2~c_3}~\LL{z}{a}}.$$
  By induction, we may then apply switch moves to change the pattern of letters $b_1,c_2$ and those in $\ol{u}_i'$ into either form listed in the result of the corollary. This pair is $\qq$, and we are finished as $m'$ is even iff $m$ is even.
\end{proof}

\begin{move}\label{Move-216543-to-432165}
  The pairs
    $$ \pp = \mtrx{\LL{a}{z}~\LL{b_1~b_2}{b_2~b_1}~\LL{c_1~c_2~c_3~c_4}{c_4~c_3~c_2~c_1}~\LL{z}{a}} \mbox{ and }
	 \qq = \mtrx{\LL{a}{z}~\LL{b_1~c_4~c_1~b_2}{b_2~c_1~c_4~b_1}~\LL{c_3~c_2}{c_2~c_3}~\LL{z}{a}}$$
    are connected by switch moves: $\{b_2,c_2\},\{c_2,c_3\},\{c_1,c_3\},\{b_1,c_2\}.$
\end{move}

\begin{cor}\label{CorMove-4blocks-to-the-left}
    If $\pp$ is a pair that is piece-wise order reversing and composed of only one $4$-block and $2$-blocks (in any order), then there is a piece-wise order reversing pair $\qq$ that is composed of a $4$-block with consecutive $2$-blocks to its right.
\end{cor}

\begin{proof}
    Suppose the $4$-block is preceded by $k$ consecutive $2$-blocks, $k>0$. By performing Move \ref{Move-436521-to-214365} to the $4$-block and the $2$-block to its left. The new pair has a $4$-block with $k-1$ consecutive $2$-blocks to its left. By iterating this process, we achieve $k=0$, or the pair $\qq$ as desired.
\end{proof}

\begin{move}\label{Move-43216587-to-21438765}
    The pairs
	$$ \pp = \mtrx{\LL{a}{z}~\LL{b_1~b_2~b_3~b_4}{b_4~b_3~b_2~b_1}~\LL{c_1~c_2}{c_2~c_1}~\LL{c_3~c_4}{c_4~c_3}~\LL{z}{a}} \mbox{ and }
	    \qq = \mtrx{\LL{a}{z}~\LL{b_1~b_4}{b_4~b_1}~\LL{b_3~b_2}{b_2~b_3}~\LL{c_1~c_4~c_3~c_2}{c_2~c_3~c_4~c_1}~\LL{z}{a}}$$
    are connected by moves $\{b_2,c_1\},\{b_3,c_4\},\{b_1,c_3\},\{b_3,c_2\},\{b_4,c_4\},\{b_2,c_3\},\{b_1,c_1\}$.
\end{move}

\begin{move}\label{Move-43218765-to-21436587}
    The pairs
	$$ \pp = \mtrx{\LL{a}{z}~\LL{b_1~b_2~b_3~b_4}{b_4~b_3~b_2~b_1}~\LL{c_1~c_2~c_3~c_4}{c_4~c_3~c_2~c_1}~\LL{z}{a}} \mbox{ and }
	    \qq = \mtrx{\LL{a}{z}~\LL{b_1~b_4}{b_4~b_1}~\LL{b_3~b_2}{b_2~b_3}~\LL{c_1~c_4}{c_4~c_1}~\LL{c_3~c_2}{c_2~c_3}~\LL{z}{a}}$$
    are connected by the following switch moves: $\{b_3,c_3\}$, $\{b_4,c_1\}$, $\{b_2,c_2\}$, $\{b_1,c_4\}$, $\{b_3,c_1\}$, $\{b_1,c_2\}$, $\{b_4,c_3\}$.
\end{move}

\begin{move}\label{Move-3241-to-2143}
    The pairs
	$$ \pp = \mtrx{\LL{a}{z} ~ \LL{b_1~c_1~c_2~b_2}{c_2~c_1~b_2~b_1}~\LL{z}{a}} \mbox{ and }
	    \qq = \mtrx{\LL{a}{z}~ \LL{c_2~b_2}{b_2~c_2} ~\LL{b_1~c_1}{c_1~b_1}~\LL{z}{a}}$$
    are connected by a $\{b_2,c_1\}$-switch.
\end{move}

\begin{cor}\label{CorMove-aUb-Uba}
    Suppose
      $$ \pp = \mtrx{\LL{a}{z}~\LL{b_1~\ol{u}_1~b_2}{\ol{u}_2~b_2~b_1} ~\LL{z}{a}}$$
    where $\ol{u}_1,\ol{u}_2$ form a chain composed of $k$ consecutive $2$-blocks. There exists $\qq$ connected to $\pp$ by switch moves such that $\qq$ is piece-wise order reversing and is composed of $k+1$ consecutive $2$-blocks.
\end{cor}

\begin{proof}
    By induction on $k$. If $k=0$, then $\pp=\qq$. If $k>0$, let $\ol{u}_1 = c_1~c_2~\ol{u}'_1$ and $\ol{u}_2 = c_2~c_1~\ol{u}'_2$. Then $\pi$ is of the form
	$$ \pp = \mtrx{\LL{a}{z}~\LL{b_1~c_1~c_2~\ol{u}'_1~b_2}{c_2~c_1~\ol{u}'_2~b_2~b_1} ~\LL{z}{a}}.$$
    By performing Move \ref{Move-3241-to-2143}, we arrive at
	$$ \pp' = \mtrx{\LL{a}{z}~\LL{c_2~\ol{u}'_1~b_2}{\ol{u}'_2~b_2~c_2}~\LL{b_1~c_1}{c_1~b_1}~\LL{z}{a}}.$$
    We may then apply our result letters in $\ol{u}'_1$ and $b_2,c_2$ to arrive at $\qq$.
\end{proof}

\section{Switch Moves Concerning \texorpdfstring{$3$}{3}-blocks As Well}

\begin{move} \label{Move-52431-to-21543}
 The pairs
     $$ \pp = \mtrx{\LL{a}{z}~\LL{b_1~b_2~c_1~c_2~b_3}{b_3~b_2~c_2~c_1~b_1}~\LL{z}{a}} \mbox{ and }
	\qq = \mtrx{\LL{a}{z} ~\LL{c_2~b_1}{b_1~c_2}~\LL{c_1~b_2~b_3}{b_3~b_2~c_1}~\LL{z}{a}}$$
 are connected by switch moves $\{b_2,c_2\},\{b_1,c_1\}$.
\end{move}

\begin{cor} \label{CorMove-abUc-cbUa}
    Suppose $\pp$ is a pair of the form
	$$ \pp = \mtrx{\LL{a}{z}~\LL{b_1~b_2~\ol{u}_1~b_3}{b_3~b_2~\ol{u}_2~b_1}~\LL{z}{a}}$$
    where $\ol{u}_1,\ol{u}_2$ form a chain composed of $m$ consecutive $2$-blocks, $m\geq 0$. Then there exists $\qq$ connected to $\pp$ by switch moves such that $\sigma$ is piece-wise order reversing and is composed of $m$ consecutive $2$-blocks followed by a $3$-block.
\end{cor}

\begin{proof}
    If $m=0$, then $\qq=\pp$.

    If $m>1$, let $\ol{u}_1 = \ol{u}_1' c_1 c_2$ and $\ol{u}_2 = \ol{u}_2' c_2 c_1$. We then express $\pp$ as
	$$ \pp = \mtrx{\LL{a}{z}~\LL{b_1~b_2~\ol{u}_1'~c_1~c_2~b_3}{b_3~b_2~\ol{u}_2'~c_2~c_1~b_1}~\LL{z}{a}},$$
    where $\ol{u}_1'$ and $\ol{u}_2'$ form $m-1$ consecutive $2$-blocks. By performing Move \ref{Move-52431-to-21543} to $\pi$, we arrive at
	$$ \pp' = \mtrx{\LL{a}{z} ~\LL{c_2~b_1~\ol{u}_1'}{b_1~\ol{u}_2'~c_2}~\LL{c_1~b_2~b_3}{b_3~b_2~c_1}~\LL{z}{a}}.$$
    By considering only the letters in $\ol{u}'_i$ as well as $b_1$ \& $c_2$, we may now apply Corollary \ref{CorMove-abU-bUa} to achieve pair $\qq$. This pair has those letters form $m$ consecutive $2$-blocks while ending with the $3$-block from letters $b_2,b_3$ and $c_1$.
\end{proof}

\begin{move}\label{Move-43215876-to-21435876}
  The pairs
    $$ \pp = \mtrx{\LL{a}{z}~\LL{b_1~b_2~b_3~b_4}{b_4~b_3~b_2~b_1}~\LL{s}{s}~\LL{c_1~c_2~c_3}{c_3~c_2~c_1}~\LL{z}{a}}\mbox{ and }
	\qq = \mtrx{\LL{a}{z}~\LL{b_1~b_4}{b_4~b_1}~\LL{b_3~b_2}{b_2~b_3}~\LL{s}{s}~\LL{c_3~c_2~c_1}{c_1~c_2~c_3}~\LL{z}{a}}$$
  are connected by moves $\{b_3,c_3\},\{c_2,s\},\{b_2,b_4\},\{c_1,c_3\},\{b_1,c_2\},\{b_3,c_1\},\{b_2,c_2\}.$
\end{move}

\begin{move}\label{Move-3217654-to-3215476}
	The pairs
	$$ \pp = \mtrx{\LL{a}{z}~\LL{b_1~b_2~b_3}{b_3~b_2~b_1}~\LL{c_1~c_2~c_3~c_4}{c_4~c_3~c_2~c_1}~\LL{z}{a}}\mbox{ and }
		\qq = \mtrx{\LL{a}{z}~\LL{b_3~b_2~b_1}{b_1~b_2~b_3}~\LL{c_1~c_4}{c_4~c_1}~\LL{c_3~c_2}{c_2~c_3}~\LL{z}{a}}$$
	are connected by moves $\{b_1,c_2\},\{b_2,c_1\},\{c_1,c_3\},\{b_1,b_3\},\{c_2,c_4\},\{b_2,c_3\}.$
\end{move}

\begin{move}\label{Move-4321765-to-2143765}
	The pairs
	$$ \pp = \mtrx{\LL{a}{z}~\LL{b_1~b_2~b_3~b_4}{b_4~b_3~b_2~b_1}~\LL{c_1~c_2~c_3}{c_3~c_2~c_1}~\LL{z}{a}}\mbox{ and }
		\qq= \mtrx{\LL{a}{z}~\LL{b_3~b_1}{b_1~b_3}~\LL{c_1~b_4}{b_4~c_1}~\LL{c_3~c_2~b_2}{b_2~c_2~c_3}~\LL{z}{a}}$$
	are connected by moves $\{b_1,c_3\},\{b_4,c_2\},\{b_1,b_2\},\{c_1,c_3\},\{b_3,c_1\}.$
\end{move}

\begin{move}\label{Move-321654-to-213654}
 The pairs
    $$ \pp = \mtrx{\LL{a}{z}~\LL{b_1~b_2~b_3}{b_3~b_2~b_1}~\LL{c_1~c_2~c_3}{c_3~c_2~c_1}~\LL{z}{a}}\mbox{ and }
	\qq = \mtrx{\LL{a}{z}~\LL{b_2~c_3}{c_3~b_2}~\LL{c_1}{c_1}~\LL{b_3~c_2~b_1}{b_1~c_2~b_3}~\LL{z}{a}}$$
    are connected by switch moves $\{b_2,c_2\},\{b_3,c_1\},\{b_1,c_3\}$.
\end{move}

\begin{move}\label{Move-3214765-moving-leftblocks}
 The pairs
    $$ \pp = \mtrx{\LL{a}{z}~\LL{b_1~b_2~b_3}{b_3~b_2~b_1}~\LL{s}{s}~\LL{c_1~c_2~c_3}{c_3~c_2~c_1}~\LL{z}{a}}\mbox{ and }
	\qq = \mtrx{\LL{a}{z}~\LL{c_3~b_2~c_1}{c_1~b_2~c_3}~\LL{s}{s}~\LL{b_1~c_2~b_3}{b_3~c_2~b_1}~\LL{z}{a}}$$
  are connected by switch moves $\{b_2,s\},\{b_1,c_2\},\{b_3,c_3\},\{b_1,c_1\}$.
\end{move}

\begin{move}\label{Move-35421-to-32154}
  The pairs
    $$ \pp = \mtrx{\LL{a}{z}~\LL{b_1~b_2~b_3~c_1~c_2}{b_3~c_2~c_1~b_2~b_1}~\LL{z}{a}}\mbox{ and }
	\qq = \mtrx{\LL{a}{z}~\LL{b_3~c_2~b_1}{b_1~c_2~b_3}~\LL{c_1~b_2}{b_2~c_1}~\LL{z}{a}}$$
  are connected by moves $\{b_3,c_2\},\{b_1,c_1\},\{b_2,b_3\}$.
\end{move}

\begin{cor}\label{CorMove-abcU-cUba}
    Suppose
      $$\pp = \mtrx{\LL{a}{z}~\LL{b_1~b_2~b_3~\ol{u}_1}{b_3~\ol{u}_2~b_2~b_1}~\LL{z}{a}}$$
    where $\ol{u}_1,\ol{u}_2$ form a chain of $m$ consecutive $2$-blocks. Then $\pp$ is connected to a piece-wise order reversing pair $\qq$ composed of one $3$-block followed by $m$ consecutive $2$-blocks.
\end{cor}
\begin{proof}
    We will perform induction on $m$. If $m=0$, then $\qq=\pp$.

    If $m>0$, let $\ol{u}_1 = \ol{u}_1' c_1c_2$ and $\ol{u}_2 = \ol{u}_2'c_2c_1$. We may then express $\pp$ as
      $$\pp = \mtrx{\LL{a}{z}~\LL{b_1~b_2~b_3~\ol{u}_1'~c_1~c_2}{b_3~\ol{u}_2'~c_2~c_1~b_2~b_1}~\LL{z}{a}}.$$
    By performing Move \ref{Move-35421-to-32154} to $\pp$, we arrive at
      $$ \pp' = \mtrx{\LL{a}{z}~\LL{b_3~c_2~b_1~\ol{u}_1'}{b_3~\ol{u}_2'~c_2~b_1}~\LL{c_1~b_2}{b_2~c_1}~\LL{z}{a}}.$$
    By considering the letters in $\ol{u}_i$ and $b_1,c_2,b_3$, we my by induction arrive at $\qq$.
\end{proof}

\begin{cor}\label{CorMove-abUcV-UcVba}
  Suppose
    $$\pp = \mtrx{\LL{a}{z}~\LL{b_1~b_2~\ol{u}_1~b_3~\ol{v}_1}{\ol{u}_2~b_3~\ol{v}_2~b_2~b_1}~\LL{z}{a}}$$
   where $\ol{u}_1,\ol{u}_2$ form a chain of $m$ consecutive $2$-blocks and $\ol{v}_1,\ol{v}_2$ form a chain of $k$ consecutive $2$-blocks. Then $\pp$ is connected to a piece-wise order reversing pair $\qq$ composed of $m$ consecutive $2$-blocks followed by a $3$-block followed by $k$ consecutive $2$-blocks.
\end{cor}

\begin{proof}
    By first ignoring $\ol{u}_1,\ol{u}_2$, we may apply Corollary \ref{CorMove-abcU-cUba}. By now ignoring the last $k$ consecutive $2$-blocks, we have (up to switching rows) a pair of the form in this Corollary, but with $k=0$. We express this as
      $$ \pp' = \mtrx{\LL{a}{z}~\LL{b_1~b_2~\ol{u}_1~b_3}{\ol{u}_2~b_3~b_2~b_1}~\LL{z}{a}}.$$
    If $m=0$, then $\pp'=\qq$. If $m>0$, let $\ol{u}_1 = c_1c_2\ol{u}_1'$ and $\ol{u}_2=c_2c_1\ol{u}_2'$. We then express $\pp'$ as
      $$ \pp' = \mtrx{\LL{a}{z}~\LL{b_1~b_2~c_1~c_2~\ol{u}_1'~b_3}{c_2~c_1~\ol{u}_2'~b_3~b_2~b_1}~\LL{z}{a}}.$$
    By performing a $\{b_3,c_1\}$-switch, we arrive at
      $$ \pp'' = \mtrx{\LL{a}{z}~\LL{c_2~\ol{u}_1'~b_3}{\ol{u}_2'~c_2~b_3}~\LL{b_1~b_2~c_1}{c_1~b_2~b_1}~\LL{z}{a}}.$$
    This will either be $\qq$, or we may apply Corollary \ref{CorMove-aUb-Uba} to arrive at $\qq$.
\end{proof}

	\bibliographystyle{abbrv}
	\bibliography{../../bibfile2}

\end{document}